\documentclass[12pt,a4paper]{article}
\usepackage{mathrsfs}
\usepackage{epsfig, graphicx}
\usepackage{latexsym,amsfonts,amsbsy,amssymb}
\usepackage{amsmath,amsthm}
\usepackage{framed}
\makeatletter 

\@addtoreset{equation}{section}
\makeatother 
\allowdisplaybreaks 

\newtheorem{theorem}{Theorem}[section]
\newtheorem{lemma}{Lemma}[section]
\newtheorem{corollary}{Corollary}[section]

\newtheorem{proposition}{Proposition}[section]

\begin{document}
\title{ A Multilevel Correction Type of Adaptive Finite Element Method for Eigenvalue Problems}
\date{September 30, 2011}
\author{
 Hehu Xie\footnote{LSEC, Institute
of Computational Mathematics, Academy of Mathematics and Systems
Science, Chinese Academy of Sciences, Beijing 100190,
China(hhxie@lsec.cc.ac.cn)}}
\maketitle
\begin{abstract}
A type of adaptive finite element method for the eigenvalue problems is proposed
based on the multilevel correction scheme. In this method, adaptive finite element
 method to solve eigenvalue problems involves solving associated boundary value problems
 on the adaptive partitions and small scale eigenvalue problems on the coarsest partitions.
 Hence the efficiency of solving eigenvalue problems can be improved to be similar to the
 adaptive finite element method for the associated boundary value problems.
 The convergence and optimal complexity is theoretically verified and numerically demonstrated.
\end{abstract}

{\bf keywords.}\ Eigenvalue problem, multilevel correction, adaptive finite element method, convergence,
optimality


\section{Introduction}
The finite element method is one of the widely used discretization schemes for solving eigenvalue problems.
The adaptive finite element method (AFEM) is a meaningful approach which can generate a sequence of optimal triangulations by refining those elements where the errors, as the local error estimators indicate, are relatively large. The
AFEM is really an effective way to make efficient use of given computational resources.
 Since Babu\v{s}ka and Rheinboldt
\cite{BabuskaRheinboldt}, the AFEM has been an active topic, many researchers are attracted to study the AFEM (see, e.g., \cite{ArnoldMukherjeePouly,BabuskaRheinboldt,BabuskaVogelius,BartelsCarstensen, Carstensen,ChenNochetto,MorinNochettoSiebert_2002,Nochetto,WuChen,YanZhou} and the references cited therein) in the last 30 years. So far, the convergence and optimality of the AFEM for boundary value problems has been obtained and
understood well (see, e.g.,  \cite{Dofler,DoflerWilderotter,MorinNochettoSiebert_2000,MorinNochettoSiebert_2002,Veeser,BinevDahmenDeVore,
MekchayNochetto,Stevenson_2007,Stevson_2008,CasconKreuzerNochettoSiebert,CarstensenHoppe,ChenHolstXu,CarstensenBartels} and the references cited therein).

Besides for the boundary value problems, the AFEM is also a very useful and efficient
way for solving eigenvalue problems (see, e.g., \cite{BeckerRannacher,ChenGongHeZhou, DuranPadraRodrguez,HeZhou,HeuvelineRannacher,Larson,MaoShenZhou,Verfurth}). The AFEM for eigenvalue
problems has been analyzed in some papers (see, e.g., \cite{DaiXuZhou,GianiGraham,HeuvelineRannacher} and the reference cited therein). Especially, \cite{DaiXuZhou,HeZhou} give an elaborate analysis of the convergence and optimality for the adaptive finite element eigenvalue computation. In \cite{GianiGraham}, authors also give the analysis of the convergence for the eigenvalue problems by the AFEM.

The purpose of this paper is to propose and analyze a type of AFEM
to solve the eigenvalue problems based on the recent work on the multilevel correction method
(see \cite{LinXie}) and  the two-grid correction method (see \cite{XuZhou_2001}).
In this new scheme, the cost of solving  eigenvalue problems is almost the same as solving the associated boundary value problems.
Here, we adopt the techniques in \cite{DaiXuZhou,HeZhou,CasconKreuzerNochettoSiebert} to prove the convergence and optimal complexity of the new AFEM for the eigenvalue problems. Our analysis is also
based on the relationship between the finite element eigenvalue approximation and the associated boundary value problem approximation (c.f. \cite{DaiXuZhou,HeZhou}).

The rest of the paper is arranged as follows. In Section 2, we shall describe some basic notation and the AFEM for
the second order elliptic problems. In Section 3, we introduce a type of AFEM for the second order elliptic
 eigenvalue problems based on the multilevel correction scheme.
The convergence analysis of this type of AFEM for eigenvalue problems will be given in Section 4 and Section 5 is devoted to proving the corresponding optimal complexity. In Section 6, some numerical experiments are presented to test the theoretical analysis. Finally, some concluding remarks are given in the last section.

\section{Preliminaries}
In this section, we introduce some basic notation and some useful results of AFEM for the second order elliptic
boundary value problem.

Let $\Omega\subset \mathcal{R}^d\ \ (d\geq 1)$ denotes a polytopic bounded domain. We use the standard notation for Sobolev space $W^{s,p}(\Omega)$ and their associated norms and seminorms  (see e.g, \cite{Adams,Ciarlet}). We  denote $H^{s}(\Omega)=W^{s,2}(\Omega)$ and $H_0^{1}(\Omega)=\big\{v\in H^{1}(\Omega):v|_{\partial \Omega}=0\big\}$, where $v|_{\partial \Omega}$ is understood in the sense of trace, $\|v\|_{s,\Omega}=\|v\|_{s,2,\Omega}$ and $\|v\|_{0,\Omega}=\|v\|_{0,2,\Omega}$. For simplicity, following Xu \cite{Xu}, we use the symbol $\lesssim$ in this paper. The notation $A\lesssim B$ means that $A\leq CB$ for some constant $C$ independent of the mesh sizes. Throughout this paper, we shall use $C$ to denote a generic positive constant, independent of the mesh sizes,
  which may varies at its different occurrences.  We consider finite element discretization on the shape regular family of nested conforming meshes $\{\mathcal{T}_h\}$ over $\Omega$: there exists a constant $\gamma*$ such that
$$\frac{h_{T}}{\rho_T}\leq \gamma*\ \ \ \ \  \forall T\in\bigcup\limits_{h}\mathcal{T}_h,$$
where $h_T$ denotes the diameter of $T$ for each $T\in \mathcal{T}_h$, and $\rho_T$ is the diameter of the biggest ball
contained in $T$, $h:=\max\{h_T:T\in \mathcal{T}_h\}$. In this paper, we use $\mathcal{E}_h$ to denote the set of interior faces (edges or sides) of $\mathcal{T}_h$.

The following lemma is a result of Sobolev trace theorem (e.g. \cite{Adams,SchneiderXuZhou}).
\begin{lemma}\label{Trace_Inequality_Lemma}
 If $s>1/2$, then
\begin{eqnarray}\label{Trace_Inequality}
\|v\|_{0,\partial T}\lesssim h_T^{-1/2}\|v\|_{0,T}+h_T^{s-1/2}\|v\|_{s,T}\ \ \ \forall v\in H^s(T),\ T\in \mathcal{T}_h.
\end{eqnarray}
\end{lemma}
For $f\in L^2(\Omega)$, we define the {\bf data oscillation} (see, e.g., \cite{MekchayNochetto,MorinNochettoSiebert_2002}) by
\begin{eqnarray}\label{Oscillation_Def}
\hskip-0.3cm
{\rm osc}(f,\mathcal{T}_h)&:=& \|h(f-f_h)\|_{0,\Omega}=\left(\sum_{T\in\mathcal{T}_h}\|h_T(f-\bar{f}_T)\|_{0,T}^2\right)^{1/2},
\end{eqnarray}
where $f_h$ denotes a piecewise polynomial approximation of $f$ over $\mathcal{T}_h$
and $\bar{f}_T=f_h|_T$. We will denote $\bar{f}_T$ be the $L^2$ projection of $f$ onto polynomials of some degree, which leads to the following inequality (see \cite{DaiXuZhou}):
\begin{eqnarray}\label{Osc_Sum_inequality}
\text{osc}(f_1+f_2,\mathcal{T}_h)\leq\text{osc}(f_1,\mathcal{T}_h)+\text{osc}(f_2,\mathcal{T}_h)\ \  \ \ \forall f_1,f_2\in L^2(\Omega).
\end{eqnarray}

\subsection{Short survey on linear elliptic problem with adaptive method}\label{Survey_Elliptic_Problem}
In this subsection, we shall present some basic results of AFEM for the second order elliptic
boundary value problem. 
Here, for simplicity, we consider the homogeneous boundary value problem:
\begin{eqnarray}\label{sourceproblem}
\left\{
\begin{array}{rcl}
Lu:=-\nabla (A\cdot \nabla u)+\varphi u&=&f\ \ \ {\rm in}\ \Omega,\\
u&=&0\ \ \ {\rm on}\ \partial\Omega,
 \end{array} \right.
\end{eqnarray}
where $A=(a_{ij})_{d\times d}$ is a symmetric positive definite matrix with $a_{ij}\in W^{1,\infty}(\Omega)\  (i,j=1,2\cdots d)$, and $0\leq \varphi \in L^{\infty}(\Omega)$.

The weak form of (\ref{sourceproblem}) is: Find $u\in H_{0}^{1}(\Omega)$ such that
\begin{eqnarray}\label{Source_Problem_Weak}
a(u,v)=(f,v)\ \ \ \forall v\in H_{0}^{1}(\Omega),
\end{eqnarray}
where the bounded bilinear form is defined by
\begin{eqnarray}
a(u,v)=\int_{\Omega}\big(\nabla u\cdot\nabla v+\varphi u v\big)d\Omega.
\end{eqnarray}
From the properties of $A$ and $\varphi$, the bilinear form $a(\cdot,\cdot)$ is bounded over $H_{0}^{1}(\Omega)$
$$|a(w,v)|\leq C_{a}\|w\|_{a,\Omega}\|v\|_{a,\Omega}\ \ \ \forall w,v \in H_{0}^{1}(\Omega),$$
and satisfies $$c_{a}\|w\|_{1,\Omega} \leq \|w\|_{a,\Omega} \leq C_{a} \|w\|_{1,\Omega},$$
where the energy norm $\|\cdot \|_{a,\Omega}$ is defined by $\| w \|_{a,\Omega}=\sqrt{a(w,w)} $, $c_{a}$ and $C_{a}$
are positive constants.  From these properties, it is well known that (\ref{Source_Problem_Weak}) has a unique solution $u\in H_0^1(\Omega)$
for any $f\in H^{-1}(\Omega)$.

Let $V_h\subset H_0^1(\Omega)$ be the corresponding family of nested finite element spaces  of continuous
piecewise polynomials over $\mathcal{T}_h$ of fixed degree $m\geq 1$, which vanish on the boundary of $\Omega$,
and are equipped with the same norm $\| \cdot \|_{a,\Omega}$ of space $H_0^1(\Omega)$.

Based on the finite element space $V_h$, we define the finite element scheme for (\ref{sourceproblem}): Find $u_h\in V_h$ such that
\begin{eqnarray}\label{dissource}
 a(u_h,v_h)=(f,v_h)\ \ \ \ \forall v_h\in V_h.
\end{eqnarray}
Define the Galerkin projection $R_h:H^1_0(\Omega)\rightarrow V_h$ by
\begin{eqnarray}\label{projection}
a(u-R_hu,v_h)=0\ \ \ \  \forall v_h\in V_h,
\end{eqnarray}
then we have $u_h=R_h u$ and
\begin{eqnarray}\label{Proj_Bounded}
\|R_h u\|_{a,\Omega}\leq \|u\|_{a,\Omega} \ \ \ \ \ \forall u\in H_0^1(\Omega).
\end{eqnarray}
From (\ref{Proj_Bounded}), it is easy to get the global a priori error estimates for the
finite element approximation based on the approximate properties of the finite element space $V_h$ (c.f. \cite{Ciarlet,XuZhou_2001}).

In order to simplify the notation, we introduce the quantity $\eta_a(h)$ as follows:
\begin{eqnarray*}
\eta_a(h)=\sup\limits_{f\in L^2(\Omega),\|f\|_{0,\Omega}=1}\inf\limits_{v_h\in V_h}\|L^{-1}f-v_h\|_{a,\Omega}.
\end{eqnarray*}
From \cite{BabuskaOsborn_1989,Ciarlet,SchneiderXuZhou}, it is known that $\eta_a(h)\rightarrow 0 \ \text{as} \ \ h\rightarrow 0 $ and the following propositions hold.
\begin{proposition}
\begin{eqnarray}
\|(I-R_h)L^{-1}f\|_{a,\Omega}\lesssim \eta_a(h)\|f\|_{0,\Omega}\ \ \ \ \ \forall f\in L^2(\Omega)\label{proposition1}
\end{eqnarray}
and
\begin{eqnarray}
\|u-R_h u\|_{0,\Omega}\lesssim \eta_a(h)\|u-R_h u\|_{a,\Omega}\ \ \ \ \ \forall u\in H_0^1(\Omega).\label{proposition2}
\end{eqnarray}
\end{proposition}

Next we follow the classic routine to define the a posteriori error estimator for finite element problem ($\ref{dissource}$). Let us define the element residual $\tilde{\mathcal{R}}_T(u_h)$ and the jump residual $\tilde{\mathcal{J}}_E(u_h)$ by
\begin{eqnarray}
\hskip-0.4cm\tilde{\mathcal{R}}_T(u_h):=f-Lu_h=f+\nabla\cdot(A\nabla u_h)-\varphi u_h \ \ \ {\rm in}\ \ T\in \mathcal{T}_h,\label{Residual_Error}\\
\hskip-0.4cm\tilde{\mathcal{J}}_E(u_h):=-A\nabla u_h^+\cdot\nu^+-A\nabla u_h^-\cdot\nu^-:=[[A\nabla u_h]]_E\cdot \nu_E \ \ {\rm on} \ \ E\in \mathcal{E}_h,\label{Edge_Jump}
\end{eqnarray}
where $E$ is the common side of elements $T^{+}$ and $T^-$ with outward normals $\nu^+$ and $\nu^-$, $\nu_E=\nu^-$, and
$\omega_E:=T^+\cap T^-$ that share the same side $E$.

For the element $T\in\mathcal{T}_h$, we define the local error indicator $\tilde{\eta}_h(u_h,T)$ by
\begin{eqnarray}
 \tilde{\eta}_h(u_h,T):=\left(h_T^2\|\tilde{\mathcal{R}}_T(u_h)\|_{0,T}^2
 +\sum\limits_{E\in \mathcal{E}_h,E\subset \partial T}h_E\|\tilde{\mathcal{J}}_E(u_h)\|^2_{0,E}\right)^{1/2},
\end{eqnarray}
and the error indicator for a subdomain $\omega\subset\Omega$ by
\begin{eqnarray}
 \tilde{\eta}_h(u_h,\omega):=\left(\sum_{T\in \mathcal{T}_h,T\subset w}\tilde{\eta}_h^2(u_h,T)\right)^{1/2}.
\end{eqnarray}
Thus $\tilde{\eta}_h(u_h,\Omega)$ denotes the error estimator of $\Omega$ with respect to $\mathcal{T}_h$.

Now we summarize the reliability and efficiency of the a posterior error estimator (see, e.g., \cite{MekchayNochetto,MorinNochettoSiebert_2002,Verfurth}):
\begin{lemma}(\cite{DaiXuZhou})
When $h_0$ is small enough,  there exist mesh independent constants $\tilde{C}_1,\hat{C}_2,\tilde{C}_3$ such that
\begin{eqnarray}\label{Upper_eta_Source_problem}
 \|u-u_h\|_{a,\Omega}\leq \tilde{C}_1\tilde{\eta}_h(u_h,\Omega),
\end{eqnarray}
and , for any $T\in\mathcal{T}_h$
\begin{eqnarray}\label{Lower_Bound_T}
\hat{C}_2^2\tilde{\eta}_h^2(u_h,T)-\tilde{C}_3^2 \sum\limits_{T\in\mathcal{T}_h,T\subset \omega_T} h_T^2\|\tilde{\mathcal{R}}_T(u_h)-\overline{\tilde{\mathcal{R}}_T(u_h)}\|_{0,T}^2\leq \|u-u_h\|^2_{a,\omega_T},
\end{eqnarray}
where $\omega_T$ contains all the elements that share at least a side with $T$, $\tilde{C}_1,\hat{C}_2,\tilde{C}_3>0$ depends only on the shape regularity $\gamma^*$, $C_a$ and $c_a$, $\bar{\omega}$ is the $L^2$-projection of $w$ onto polynomials of degree $m$ on $T$.

 As a consequence of (\ref{Lower_Bound_T}), we have
\begin{eqnarray}\label{Lower_Bound_Global_1}
\tilde{C}_2^2\tilde{\eta}_h^2(u_h,\Omega)-\tilde{C}_3^2{\rm osc}(f-L(u_h),\mathcal{T}_h)^2\leq \|u-u_h\|^2_{a,\Omega},
\end{eqnarray}
where $\tilde{C}_2=\frac{\hat{C}_2}{\check{C}_0}$ and $\check{C}_0$ is a constant depending
 on the shape regularity of the mesh $\mathcal{T}_h$.
\end{lemma}
It is obvious that if the right hand side term $f$ of (\ref{sourceproblem}) is a piecewise polynomial function over $\mathcal{T}_h$, (\ref{Lower_Bound_Global_1}) can be simplified to
\begin{eqnarray}\label{Lower_Bound_Global_2}
\tilde{C}_2^2\tilde{\eta}_h^2(u_h,\Omega)-\tilde{C}_3^2\text{osc}(L(u_h),\mathcal{T}_h)^2\leq \|u-u_h\|^2_{a,\Omega}.
\end{eqnarray}

There are adaptive algorithms in \cite{CasconKreuzerNochettoSiebert,Dofler,MekchayNochetto,MorinNochettoSiebert_2002} to solve (\ref{dissource}), which introduce two type of marking strategies to promise reduction of both error and oscillation. For these two type of methods, both convergence and optimal complexity of the adaptive finite element algorithm have been obtained (see, e.g., \cite{Dofler,MekchayNochetto,MorinNochettoSiebert_2002,Stevenson_2007,Stevson_2008}). However, the oscillation marking is not necessary which has been proved by Cascon et al. \cite{CasconKreuzerNochettoSiebert}. Thus, the adaptive algorithm without oscillation marking which is adopted in this paper can be stated as follows (c.f. \cite{CasconKreuzerNochettoSiebert,DaiXuZhou}).

\begin{framed}
\begin{center}
{\bf  Adaptive Algorithm $C_0$}
\end{center}
\noindent
Choose parameter $0<\theta<1$:\\
\noindent
1. Let $k=0$, pick an initial mesh $\mathcal{T}_{h_0}$ and start the loop.\\
2. On the  mesh $\mathcal{T}_{h_k}$, solve the problem (\ref{dissource}) for the discrete solution $u_{h_k}$.\\
3. Compute the local indicators $\tilde{\eta}_{h_k}(u_h,T)$.\\
4. Construct the submesh $\widehat{\mathcal{T}}_{h_k}\subset \mathcal{T}_{h_k}$ by
\textbf{Marking Strategy $E_0$} with parameters $\theta$.\\
5. Refine $\mathcal{T}_{h_k}$ to generate a new conforming mesh $\mathcal{T}_{h_{k+1}}$ by procedure {\bf REFINE}.\\
6. Let $k=k+1$ and go to step 2.\\
\end{framed}
As in \cite{CasconKreuzerNochettoSiebert}, the procedure {\bf REFINE} used in {\bf Adaptive Algorithm $C_0$} is not required  to satisfy the Interior Node Property of \cite{MekchayNochetto,MorinNochettoSiebert_2002}.  Here we use the iterative or recursive
 bisection (see, e.g., \cite{Maubach,Traxler}) of elements with the minimal refinement condition in the procedure
 \textbf{REFINE}. The \textbf{Marking Strategy} adopted in \textbf{Adaptive Algorithm $C_0$} was introduced by
  D\"{o}rfler \cite{Dofler} and Morin et al. \cite{MorinNochettoSiebert_2002} and can be defined as follows.
\begin{framed}
\begin{center}
{\bf Marking Strategy $E_0$}
\end{center}

\noindent Given parameter $0<\theta<1$:\\
\noindent
1. Construct a minimal subset $\widehat{\mathcal{T}}_H$ from $\mathcal{T}_H$ by
selecting some elements in $\mathcal{T}_H$ such that
\begin{eqnarray*}
 \sum\limits_{T\in \widehat{\mathcal{T}}_H}\tilde{\eta}^2_H(u_h,T)\geq\theta\tilde{\eta}^2_H(u_h,\Omega).
\end{eqnarray*}
2. Mark all the elements in $\widehat{\mathcal{T}}_H$.
\end{framed}

In order to analyze the convergence of the AFEM, we need the following lemma.
\begin{lemma}(\cite{DaiXuZhou})\label{Osc_Bounded_Lemma}
 There exits a constant $C_*$ only depending on the equation parameters and the mesh regularity $\gamma^*$ such that
\begin{eqnarray}\label{Osc_Bounded}
 {\rm osc}(L(v_h),\mathcal{T}_h)\leq C_L\|v_h\|_{a,\Omega}\ \  \ \ \forall v_h\in V_h.
\end{eqnarray}
\end{lemma}
The convergence of \textbf{Adaptive Algorithm $C_0$} has been proved by
Cascon et al \cite{CasconKreuzerNochettoSiebert} and can be stated as follows.
\begin{theorem}(\cite{CasconKreuzerNochettoSiebert})\label{Adaptive_Convergence_Source_Theorem}
 Let $\{u_{h_k}\}_{k\in \mathbb N_0}$ be a sequence finite element solutions of
 (\ref{sourceproblem}) based on the sequence of nested meshes $\{\mathcal{T}_{h_k}\}_{k\in\mathbb N_0}$ produced by
 {\bf Adaptive Algorithm $C_0$}.
Then, there exist constants $\tilde{\gamma}$ and $\xi\in (0,1)$, depending on the shape
regularity of meshes, the data and the parameters used in {\bf Adaptive Algorithm $C_0$}, such that
 any two consecutive iterates $k$ and $k+1$ have the property
\begin{eqnarray}\label{Adaptive_Convergence_Source}
 \|u-u_{h_{k+1}}\|^2_{a,\Omega}+\tilde{\gamma}\tilde{\eta}^2_{h_{k+1}}(u_{h_{k+1}},\Omega)\leq \xi^2
 \Big( \|u-u_{h_{k}}\|^2_{a,\Omega}+\tilde{\gamma}\tilde{\eta}^2_{h_{k}}(u_{h_{k}},\Omega)\Big),
\end{eqnarray}
where $\mathbb N_0=\{0, 1, 2, 3, \cdots\}$ and the constant $\tilde{\gamma}$ has the following form
\begin{eqnarray}\label{Def_tilde_Gamma_source}
\tilde{\gamma}=\frac{1}{(1+\delta^{-1})C^2_L},
\end{eqnarray}
with some constant $\delta\in (0, 1)$.

\end{theorem}

\section{The  eigenvalue problem and adaptive finite element method based on multilevel correction}
In this section, we introduce a type of AFEM based on multilevel correction scheme for the linear
second order elliptic eigenvalue problem.

We are concerned with the following  eigenvalue problem
\begin{eqnarray}\label{eigenproblem}
\left\{
\begin{array}{rcl}
-\nabla (A\cdot \nabla u)+\varphi u&=&\lambda u\ \ \ {\rm in}\ \Omega,\\
u&=&0\ \ \ \ \ {\rm on}\ \partial\Omega,\\
\int_{\Omega} u^2d\Omega&=&1.
 \end{array} \right.
\end{eqnarray}

The corresponding weak form can be written as:
Find $(\lambda, u)\in \mathcal{R}\times H_0^1(\Omega)$ such that $\|u\|_{0,\Omega}=1$ and
\begin{eqnarray} \label{eigenvalue problem}
a(u,v)=\lambda(u,v)\ \ \ \  \ \forall v\in H^1_0(\Omega).
\end{eqnarray}
As we know the eigenvalue problem (\ref{eigenvalue problem}) has a countable sequence of real eigenvalues
$$0<\lambda_1<\lambda_2\leq\lambda_3\leq\cdots$$
and corresponding orthogonal eigenfunctions
$$u_1,u_2,u_3,\cdots, $$
which satisfy  $(u_i,u_j)=\delta_{ij}, i,j=1,2,\cdots$.

Now we state an useful Rayleigh quotient expansion of the eigenvalue which is expressed by the
 eigenfunction approximation (see \cite{BabuskaOsborn_1989,LinXie,XuZhou_2001}).
\begin{lemma}\label{Expansion_Eigenvalue_Lemma}
Let $(\lambda,u)$ be an eigenpair of (\ref{eigenvalue problem}). Then for any
$w \in H_0^1(\Omega)$, we have
\begin{eqnarray}\label{Expansion_Eigenvalue}
 \frac{a(w,w)}{(w,w)}-\lambda =\frac{a(w-u,w-u)}{(w,w)}-\frac{\lambda(w-u,w-u)}{(w,w)}.
\end{eqnarray}
\end{lemma}
The standard finite element discretization for (\ref{eigenvalue problem}) is:
Find $(\lambda_h,u_h)\in \mathcal{R}\times V_h$ such that $\|u_h\|_0=1$ and
\begin{eqnarray}\label{diseig}
 a(u_h,v_h)=\lambda_h(u_h,v_h)\ \ \ \forall v_h\in V_h.
\end{eqnarray}
We can also order the eigenvalues of (\ref{diseig}) as an  increasing sequence
$$0<\lambda_{1,h} < \lambda_{2,h}\leq\lambda_{3,h}\leq\cdots\leq \lambda_{n_h,h},\ \ \ n_h={\rm dim}V_h,$$
and the corresponding orthogonal eigenfunctions
$$u_{1,h},\ u_{2,h}, \ u_{3,h},\ \cdots,\ u_{n_h,h} $$
satisfying $(u_{i,h},u_{j,h})=\delta_{ij}$, $i, j=1, 2, \cdots, n_h$.

From the minimum-maximum principle (see \cite{BabuskaOsborn_1989,Chatelin}) and Lemma \ref{Expansion_Eigenvalue_Lemma}, we have
\begin{eqnarray}
 \lambda_i\leq\lambda_{i,h}\leq\lambda_i+\bar{C}_i\|u_i-u_{i,h}\|_{a,\Omega}^2,\ \ i=1,2,,\cdots,n_h
\end{eqnarray}
with $\bar{C}_i$ constants independent of mesh size $h$.

\subsection{Adaptive multilevel correction algorithm for eigenvalue problem}
The adaptive procedure consists of loops of the form

\begin{center}
\textbf{Solve $\rightarrow$ Estimate $\rightarrow$ Mark $\rightarrow$ Refine.}
\end{center}

Similarly to {\bf Marking Strategy $E_0$}, we define {\bf Marking Strategy $E$} for (\ref{diseig}) to
enforce the error reduction as follows:
\begin{framed}
\begin{center}
{\bf Marking Strategy $E$}
\end{center}

\noindent Given a parameter $0<\theta<1$:\\
\noindent
1. Construct a minimal subset $\widehat{\mathcal{T}}_H$ from $\mathcal{T}_H$ by
selecting some elements in $\mathcal{T}_H$ such that
\begin{eqnarray*}
 \sum\limits_{T\in \widehat{\mathcal{T}}_H}\eta^2_H(u_h, T)\geq\theta\eta^2_H(u_h,\Omega),
\end{eqnarray*}
where $\eta_h(u_h,T)$ and $\eta_h(u_h,\Omega)$ denote the error indicator of the eigenfunction
approximation  $u_h$ on $T$ and $\Omega$, respectively.

\noindent
2. Mark all the elements in $\widehat{\mathcal{T}}_H$.
\end{framed}

Then we present a type of AFEM to compute the eigenvalue problem in the multilevel
correction framework.
\begin{framed}
\begin{center}
\textbf{Adaptive Algorithm $C$ }
\end{center}

\noindent 1. Pick up an initial mesh $\mathcal{T}_{h_{0}}$ with mesh size $h_0$.\\
2. Construct the finite element space $V_{h_0}$ and solve the following
eigenvalue problem to get the discrete solution $(\lambda_{h_{0}},u_{h_0})\in \mathcal{R}\times V_{h_0}$ such that $\|u_{h_0}\|_{0,\Omega}=1$ and
\begin{eqnarray}\label{Eigen_Dis_h_0}
a(u_{h_0}, v_{h_0})&=&\lambda_{h_0}(u_{h_0}, v_{h_0})\ \ \ \ \ \forall v_{h_0}\in V_{h_0}.
\end{eqnarray}
3. Let $k=0$.\\
4. Compute the local error indicators  $\eta_{h_{k}}(u_{h_k}, T)$.\\
5. Construct $\widehat{\mathcal{T}}_{h_{k}}\subset \mathcal{T}_{h_{k}}$ by {\bf Marking Strategy E} and parameter $\theta$.\\
6. Refine $\mathcal{T}_{h_{k}}$ to get a new conforming mesh $\mathcal{T}_{h_{k+1}}$ by procedure \textbf{Refine}.\\
7. Solve the following source problem on $\mathcal{T}_{h_{k+1}}$ for the discrete solution $\widetilde{u}_{h_{k+1}}\in V_{h_{k+1}}$:
\begin{eqnarray}\label{BVP_Dis_h_k+1}
a(\widetilde{u}_{h_{k+1}},v_{h_{k+1}})&=&\lambda_{h_k}(u_{h_k},v_{h_{k+1}})\ \ \ \ \ \forall v_{h_k}\in V_{h_k}.
\end{eqnarray}
8. Construct the new finite element space $V_{h_0,h_{k+1}}=V_{h_0}+{\rm span} \{\widetilde{u}_{h_{k+1}}\}$ and solve the eigenvalue problem to get the solution $(\lambda_{h_{k+1}}, u_{h_{k+1}})\in
\mathcal{R}\times V_{h_0,h_{k+1}}$ such that $\|u_{h_{k+1}}\|_{0,\Omega}=1$ and
\begin{eqnarray}\label{Eigen_Dis_h_k+1}
a(u_{h_{k+1}}, v_{h_{h_0,h_{k+1}}})=\lambda_{h_{k+1}}(u_{h_{k+1}}, v_{h_0,h_{k+1}})\ \ \  \ \forall v_{h_0,h_{k+1}}\in V_{h_0,h_{k+1}}.
\end{eqnarray}
9. Let $k=k+1$ and go to Step 4.
\end{framed}
Local error indicator $\eta_h(u_{h_k},T)$ in {\bf Adaptive Algorithm $C$} will be given in the next subsection.
For the aim of error estimate, we define
\begin{eqnarray*}
M(\lambda_i)&=&\big\{w\in H_0^1(\Omega):w \text{ is an eigenfunction of (\ref{eigenproblem}) }\\
&& \quad\quad\quad \ \ \ \ \ \ \ \ \ \text{ corresponding to the eigenvalue }\lambda_i\big\}
 \end{eqnarray*}
and the quantity
$$\delta_h(\lambda_i)=\sup\limits_{w\in M(\lambda_i),\|w\|_{0,\Omega}=1}\inf\limits_{v_h\in V_h}\|w-v_h\|_{a,\Omega}.$$

In the following analysis, we only need some crude priori error estimates stated as follows.
\begin{lemma}\label{Error_Estimate_Crude_Lemma}
The obtained eigenpair approximation $(\lambda_{h_k},u_{h_k})\ (k=0, 1,\cdots)$ after each adaptive step in
 {\bf Adaptive Algorithm $C$} has the error estimate
\begin{eqnarray}
\|u-u_{h_k}\|_{a,\Omega}&\lesssim & \delta_{h_0}(\lambda),\label{Error_u_u_h_k}\\
\|u-u_{h_k}\|_{0,\Omega}&\lesssim& \eta_a(h_0)\|u-u_{h_k}\|_{a,\Omega},\label{Error_u_u_h_k_-1}\\
|\lambda-\lambda_{h_k}| &\lesssim& \|u-u_{h_k}\|_{a,\Omega}^2. \label{Error_lambda_h_k}
\end{eqnarray}
\end{lemma}
\begin{proof}
Based on the error estimate theory of eigenvalue problem by finite
element method (c.f. \cite{BabuskaOsborn_1989,BabuskaOsborn_1991}),
the eigenfunction approximation of problem (\ref{Eigen_Dis_h_0}) or (\ref{Eigen_Dis_h_k+1})  has the following estimates
\begin{eqnarray}\label{Error_u_u_h_2}
\|u-u_{h_k}\|_{a,\Omega}&\lesssim& \sup_{w\in M(\lambda)}\inf_{v_h\in
V_{h_0,h_k}}\|w-v_h\|_{a,\Omega} \nonumber\\
&\lesssim& \sup_{w\in M(\lambda)}\inf_{v_h\in
V_{h_0}}\|w-v_h\|_{a,\Omega}\lesssim \delta_{h_0}(\lambda).
\end{eqnarray}
and
\begin{eqnarray}\label{Error_u_u_h_2_Negative}
\|u-u_{h_k}\|_{0,\Omega}&\lesssim&\widetilde{\eta}_a(h_0)\|u-u_{h_k}\|_{a,\Omega},
\end{eqnarray}
where
\begin{eqnarray}\label{Eta_a_h_2}
\widetilde{\eta}_a(h_0)&=&\sup_{f\in V,\|f\|_{0,\Omega}=1}\inf_{v\in
V_{h_0,h_k}}\|L^{-1}f-v\|_{a,\Omega} \leq \eta_a(h_0).
\end{eqnarray}
From (\ref{Error_u_u_h_2}), (\ref{Error_u_u_h_2_Negative}), and (\ref{Eta_a_h_2}), we can obtain
(\ref{Error_u_u_h_k}) and (\ref{Error_u_u_h_k_-1}). The
estimate (\ref{Error_lambda_h_k}) can be derived by Lemma \ref{Expansion_Eigenvalue_Lemma} and (\ref{Error_u_u_h_k_-1}).
\end{proof}

\subsection{A posteriori error estimate for eigenvalue problem}
Now, we are going to give an a posteriori error estimator for the eigenvalue problem.
The a posteriori error estimators
have been studied extensively (see, e.g., \cite{BeckerRannacher,DaiXuZhou,DuranPadraRodrguez,HeZhou,HeuvelineRannacher,Larson,MaoShenZhou}).
Here we use the similar way in \cite{DaiXuZhou} to derive the a posteriori error estimator for
 the eigenvalue problem by {\bf Adaptive Algorithm $C$}  
 from a relationship between the elliptic eigenvalue approximation and
 the associated boundary value approximation.


In this paper, we set $H=h_k$ and $h=h_{k+1}$.
Let $K:L^2(\Omega)\rightarrow H_0^1(\Omega)$ be the operator defined by
\begin{eqnarray}\label{postK}
a(Kw,v)=(w,v)\ \ \ \ \ \ \forall w, v\in H_0^1(\Omega).
\end{eqnarray}
Then the eigenvalue problems (\ref{eigenvalue problem}) and (\ref{diseig}) can be written as
\begin{eqnarray}
 u=\lambda Ku,\ \ \ \ u_h=\lambda_hR_hKu_h.
\end{eqnarray}
In the step $7$ of {\bf Adaptive Algorithm $C$}, we can view (\ref{BVP_Dis_h_k+1})
 as the finite element approximation of the problem
\begin{eqnarray}
 a(w^{H},v)&=&(\lambda_{H} u_{H},v)\ \ \ \  \forall v\in H_0^1(\Omega).
\end{eqnarray}
Thus we have $w^{H}=\lambda_{H}K u_{H}$ and
\begin{eqnarray}\label{postGar}
 \widetilde{u}_h=R_{h} w^{H}.
\end{eqnarray}
Similarly we can also define $w^h$ as
\begin{eqnarray}\label{w_h_k_defintion}
a(w^h,v)&=&(\lambda_hu_h,v)\ \ \ \ \ \forall v\in H_0^1(\Omega).
\end{eqnarray}

 Let $r(h_0):=\eta_a(h_0)+\|u-u_H\|_{a,\Omega}$.
 Obviously from Lemma \ref{Error_Estimate_Crude_Lemma}, we know $r(h_0)\ll 1$ if $h_0$ is small enough.
\begin{theorem}\label{trans}
We have the following  estimate
\begin{eqnarray}\label{Estimate_u_u_h}
\|u-u_h\|_{a,\Omega}=\|w^h-R_hw^{h}\|_{a,\Omega}+\mathcal{O}(r(h_0))(\|u-u_H\|_{a,\Omega}+\|u-u_h\|_{a,\Omega}).
\end{eqnarray}
\end{theorem}
\begin{proof}
 From the definition (\ref{Eigen_Dis_h_k+1}), we have
\begin{eqnarray}\label{u_u_h}
u-u_h&=&u-w^h+w^h-R_hw^h+R_hw^h-R_hw^H\nonumber\\
&&+R_hw^H-u_h\nonumber\\
&=&u-w^h+w^h-R_hw^h+R_h(w^h-w^H)\nonumber\\
&&+\widetilde{u}_h-u_h.
\end{eqnarray}
The following equality holds
\begin{eqnarray}
u-w^h &=&\lambda Ku-\lambda_hKu_h =\lambda K(u-u_h)+(\lambda-\lambda_h)Ku_h,
\end{eqnarray}
which together with the fact $\|K(u-u_h)\|_{a,\Omega}\lesssim \|u-u_h\|_{0,\Omega}$, (\ref{Error_u_u_h_k_-1}), and
(\ref{Error_lambda_h_k}) leads to
\begin{eqnarray}\label{u_w_h}
\|u-w^h\|_{a,\Omega}&\leq&\tilde{C}r(h_0)\|u-u_h\|_{a,\Omega}.
\end{eqnarray}
Similarly, we have
\begin{eqnarray}\label{w_h_H}
\|R_hw^h-R_hw^H\|_{a,\Omega} &\leq& \|w^h-w^H\|_{a,\Omega} =\|\lambda_hKu_h-\lambda_HKu_H\|_{a,\Omega}\nonumber\\
&\leq& \tilde{C}r(h_0)(\|u-u_h\|_{a,\Omega}+\|u-u_H\|_{a,\Omega}).
\end{eqnarray}
Since $u_h-\widetilde{u}_h \in V_{h_0,h}$, (\ref{Error_u_u_h_k_-1}), and (\ref{Error_lambda_h_k}), the following
inequality holds
\begin{eqnarray}
&&a(u_h-\widetilde{u}_h,u_h-\widetilde{u}_h)=(\lambda_hu_h-\lambda_Hu_H,u_h-\widetilde{u}_h)\nonumber\\
&\leq&|\lambda_h-\lambda_H|(u_h,u_h-\widetilde{u}_h)+\lambda_H(u_h-u_H,u_h-\widetilde{u}_h)\nonumber\\
&\leq&\tilde{C}r(h_0)(\|u-u_h\|_{a,\Omega}+\|u-u_H\|_{a,\Omega})\|u_h-\widetilde{u}_h\|_{a,\Omega}.
\end{eqnarray}
Thus with the coercivity of $a(\cdot,\cdot)$,  we have
\begin{eqnarray}\label{u_tilde_u_h}
\|u_h-\widetilde{u}_h\|_{a,\Omega}&\leq& \tilde{C}r(h_0)(\|u-u_h\|_{a,\Omega}+\|u-u_H\|_{a,\Omega}).
\end{eqnarray}
Finally from (\ref{u_u_h}), (\ref{u_w_h}), (\ref{w_h_H}), and (\ref{u_tilde_u_h}),
 the desired result (\ref{Estimate_u_u_h}) can be obtained and the proof is complete.
\end{proof}

Theorem \ref{trans} builds a basic relationship between $\|u-u_h\|$ and $\|w^h-R_hw^{h}\|$, the former
is the error between the ture and the discrete eigenfunctions, while the latter the error
between $w^h$ and its finite element projection, which has been well analyzed. Since the difference between
 $\|u-u_h\|$ and $\|w^h-R_hw^{h}\|$ is a higher order term, as in \cite{DaiXuZhou},
  we follow the procedure of the analysis of convergence and complexity for the source problem.

We define the element residual $\mathcal{R}_T(u_h)$ and the jump residual $\mathcal{J}_E(u_h)$ as
follow:
\begin{eqnarray}
 \mathcal{R}_T(u_h):=\lambda_h u_h-L(u_h)=\lambda_hu_h+\nabla\cdot(A\nabla u_h)-\varphi u_h \ \ \text{in}\ T\in \mathcal{T}_h,\\
\mathcal{J}_E(u_h):=-A \nabla u_h^+ \cdot \nu^+-A \nabla u_h^- \cdot \nu^-:=[[A\nabla u_h]]_E\cdot\nu_E\ \ \ \text{on}\  E\in \mathcal{E}_h,
\end{eqnarray}
where $E$, $\nu^+$ and $\nu^-$ are defined as those of Sect. \ref{Survey_Elliptic_Problem}.

For each element $T\in \mathcal{T}_h$, we define the local error indicator $\eta_h(u_h,T)$ by
\begin{eqnarray}\label{eta_definition}
 \eta_h^2(u_h,T):=h_T^2\|\mathcal{R}_T(u_h)\|_{0,T}^2+\sum\limits_{E\in \mathcal{E}_h,E\subset \partial T}h_E\|\mathcal{J}_E(u_h)\|^2_{0,E}.
\end{eqnarray}
Then on a subset $\omega\subset \Omega$, we define the error estimator $\eta_h(u_h,\omega)$ by
\begin{eqnarray}
 \eta_h(u_h,\omega):=\left(\sum_{T\in \mathcal{T}_h,T\subset w}\eta_h^2(u_h,T)\right)^{1/2}.
\end{eqnarray}
As same as (\ref{Osc_Sum_inequality}) and (\ref{Osc_Bounded}), we have the similar inequalities of the
indicator $\eta_h(v_h,\omega)$ for any $v_h\in V_h$.
\begin{lemma}\label{Eta_Sum_Bounded_Lemma}
The following inequalities for the indicator $\eta_h(v_h,\omega)$ hold
\begin{eqnarray}
\eta_h(w_h+v_h,\omega) &\leq& \eta_h(w_h,\omega)+\eta_h(v_h,\omega)
\ \ \ \  \forall w_h \in V_h,  v_h\in V_h,\label{eta_sum_inequality}\\
\eta_h(v_h,\Omega) &\leq& C_R\|v_h\|_{a,\Omega}\ \ \ \ \forall v_h\in V_h.\label{eta_H_1_Bounded}
\end{eqnarray}
\end{lemma}
\begin{proof}
The first inequality (\ref{eta_sum_inequality}) can be obtained from the definition of $\eta_h$. Now we prove the second inequality (\ref{eta_H_1_Bounded}).

It is obvious that the inverse estimate implies
\begin{eqnarray}
\|L(v_h)\|_{0,T}&\leq&C_Ah_T^{-1}\|\nabla v_h\|_{0,T}+C_c\|v_h\|_{0,T}\ \ \ \ \forall T\in\mathcal{T}_h,
\end{eqnarray}
where $C_A$ depends on $A$ and the shape regularity constant $\gamma^*$, $C_c$ depends on the coefficient $\varphi$.
Namely, there exist some constants $\tilde{C}_T$ and $C_R$ depending on $C_A$ and $C_c$ such that
\begin{eqnarray}\label{eta_bounded_1}
\sum_{T\in\mathcal{T}_h}h_T^2\|L(v_h)\|_{0,T}^2\leq\sum_{T\in\mathcal{T}_h}\tilde{C}_T^2\|v_h\|_{1,T}^2\leq C_R^2\|v_h\|_{a,\Omega}^2.
\end{eqnarray}
From the trace inequality (\ref{Trace_Inequality}) and the inverse estimate, we have
\begin{eqnarray}\label{eta_bounded_2}
h_E\big\|[[A\nabla v_h]]_E\cdot\nu_E\big\|_{0,E}^2&\leq & \tilde{C}_E\|A\nabla v_h\|_{0,\omega_E}^2 + \tilde{C}_Eh_{T^+}^2\|A\nabla v_h\|_{1,T^+}^2\nonumber\\
&&\ \ \ \ +\tilde{C}_Eh_{T^-}^2\|A\nabla v_h\|_{1,T^-}^2\nonumber\\
&\leq& C_E^2 \|v_h\|_{1,\omega_E}^2,
\end{eqnarray}
where $\omega_E:=T^+\cup T^-$ denotes the patch including the elements sharing the edge $E$ and the constant
$C_E$ depends on $A$ and the shape regularity constant $\gamma^*$.

Hence the desired result (\ref{eta_H_1_Bounded}) can be obtained from (\ref{eta_definition}),
(\ref{eta_bounded_1}), and (\ref{eta_bounded_2}) and the proof is complete.
\end{proof}

Based on Theorem \ref{trans} and Lemma \ref{Eta_Sum_Bounded_Lemma}, the error estimate $\eta_h(u_h,\Omega)$ has
 the following properties.
\begin{theorem}
 Let $h_0$ be small enough and $h\in (0,h_0]$. Then there are mesh independent constants such that
\begin{eqnarray}\label{Upper_Bound_Eigenfunct}
\|u-u_h\|_{a,\Omega}\leq C_1\eta_h(u_h,\Omega)+\mathcal{O}(r(h_0))\|u-u_H\|_{a,\Omega},
\end{eqnarray}
and
\begin{eqnarray}\label{Lower_Bound_Eigenfunct}
C_2^2\eta_h^2(u_h,\Omega)-C_3^2 {\rm osc}(L(u_h),\mathcal{T}_h)^2 \leq  \|u-u_h\|_{a,\Omega}^2 +\mathcal{O}(r^2(h_0))\|u-u_H\|_{a,\Omega}^2.
\end{eqnarray}
Consequently we have
\begin{eqnarray}\label{Upper_Bound_Eigenvalue}
|\lambda-\lambda_h|\lesssim \eta_h^2(u_h,\Omega)+\mathcal{O}(r^2(h_0))\|u-u_H\|_{a,\Omega}^2,
\end{eqnarray}
and
\begin{eqnarray}\label{Lower_Bound_Eigenvalue}
\eta_h^2(u_h,\Omega)-{\rm osc}(L(u_h),\mathcal{T}_h)^2 \lesssim |\lambda-\lambda_h| +\mathcal{O}(r^2(h_0))\|u-u_H\|^2_{a,\Omega}.
\end{eqnarray}
\begin{proof}
From (\ref{Upper_eta_Source_problem}), (\ref{Estimate_u_u_h}), (\ref{u_u_h}), (\ref{w_h_H}), (\ref{u_tilde_u_h}), (\ref{eta_sum_inequality}), and (\ref{eta_H_1_Bounded}), we have
\begin{eqnarray}
&& \|w^h-R_hw^h\|_{a,\Omega} \leq\tilde{C}_1\eta_h(R_hw^h,\Omega)
 = \tilde{C}_1\eta_h(u_h+R_hw^h-u_h,\Omega)\nonumber\\
 &\leq&\tilde{C}_1\eta_h(u_h,\Omega)+ \tilde{C}_1\eta_h(R_hw^h-u_h,\Omega)\nonumber\\
 &\leq&\tilde{C}_1\eta_h(u_h,\Omega)+ \tilde{C}_1C_R\|R_hw^h-u_h\|_{a,\Omega}\nonumber\\
 &\leq& \tilde{C}_1\eta_h(u_h,\Omega)+ \tilde{C}_1C_Rr(h_0)(\|u-u_h\|_{a,\Omega}+\|u-u_H\|_{a,\Omega}),
\end{eqnarray}
which together with (\ref{Estimate_u_u_h}) leads to the desired result (\ref{Upper_Bound_Eigenfunct}) with
\begin{eqnarray}\label{Setting_C_1}
C_1=\tilde{C}_1.
\end{eqnarray}
Combing (\ref{Lower_Bound_Global_2}),  (\ref{Osc_Bounded}), (\ref{u_tilde_u_h}), (\ref{eta_sum_inequality}),
and (\ref{eta_H_1_Bounded}) leads to
\begin{eqnarray}\label{Lower_Bound_Relation}
&&\tilde{C}_2^2\eta_h^2(u_h,\Omega)-\tilde{C}_3^2 {\rm osc}(L(u_h),\mathcal{T}_h)^2 \nonumber\\
&\leq&\tilde{C}_2^2\tilde{\eta}_h^2(R_hw^h,\Omega)-\tilde{C}_3^2 {\rm osc}(L(R_hw^h),\mathcal{T}_h)^2\nonumber\\
&&+(C_L^2+C_R^2)(\tilde{C}_2^2+\tilde{C}_3^2)\|u_h-R_hw^h\|_{a,\Omega}^2\nonumber\\
&\leq& \|w^h-R_hw^h\|_{a,\Omega}^2
+ \tilde{C}\|u_h-R_hw^h\|_{a,\Omega}^2\nonumber\\
&\leq&2\|u-u_h\|_{a,\Omega}^2+ Cr^2(h_0)(\|u-u_h\|_{a,\Omega}^2+\|u-u_H\|_{a,\Omega}^2),
\end{eqnarray}
where the constant $C$ depends on $\tilde{C}$, $C_L$, $C_R$, $\tilde{C}_2$ and $\tilde{C}_3$.
From (\ref{Lower_Bound_Relation}), the desired result (\ref{Lower_Bound_Eigenfunct}) can obtained with
\begin{eqnarray}\label{Setting_C_2_C_3}
C_2^2 = \frac{\tilde{C}_2^2}{2+Cr^2(h_0)},\ \ \ C_3^2 = \frac{\tilde{C}_3^2}{2+Cr^2(h_0)}.
\end{eqnarray}

From Lemma \ref{Expansion_Eigenvalue_Lemma}
and (\ref{Upper_Bound_Eigenfunct}), we obtain (\ref{Upper_Bound_Eigenvalue}) and (\ref{Lower_Bound_Eigenvalue})
 can be derived from (\ref{Expansion_Eigenvalue}) and (\ref{Lower_Bound_Eigenfunct}).
\end{proof}

\end{theorem}

\section{Convergence of adaptive finite element method for eigenvalue problem}
In this section, we give the convergence analysis of the {\bf Adaptive Algorithm $C$} for the eigenvalue problem.

Before establishing the error reduction of the {\bf Adaptive Algorithm $C$} for the eigenvalue problem,
 we give some preparations.

Similarly to Theorem \ref{trans}, we also give some relationships between two level approximations, which
will be used in the following analysis.
\begin{lemma}\label{u_u_h_convergence_Lemma}
 Let $h, H\in (0,h_0]$. If $w^h=\lambda_hKu_h$, $w^H=\lambda_HKu_H$, we have
\begin{eqnarray}
 \|u-u_h\|_{a,\Omega}&=&\|w^H-R_hw^H\|_{a,\Omega}\nonumber\\
 &&\ \  +\mathcal{O}(r(h_0))\big(\|u-u_h\|_{a,\Omega}+\|u-u_H\|_{a,\Omega}\big),\label{u-u_h_H_Space}\\
{\rm osc}(L(u_h),\mathcal{T}_h)&=&{\rm osc}(L(R_hw^H),\mathcal{T}_h)\nonumber\\
&&\ \ +\mathcal{O}(r(h_0))(\|u-u_H\|_{a,\Omega}+\|u-u_h\|_{a,\Omega}), \label{osc_u_h}
\end{eqnarray}
and
\begin{eqnarray}\label{eta_u_h_R_h_w}
\eta_h(u_h,\Omega)&=&\tilde{\eta}_h(R_h w^H,\Omega)
+\mathcal{O}(r(h_0))(\|u-u_h\|_{a,\Omega}+\|u-u_H\|_{a,\Omega}).
\end{eqnarray}
\end{lemma}
\begin{proof}
First we have
\begin{eqnarray}\label{u-u_h_w_H}
u-u_h=u-w^H+w^H-R_hw^H+R_hw^H-u_h.
\end{eqnarray}
Similarly to (\ref{u_w_h}), the following inequality holds
\begin{eqnarray}\label{u_w_H_Convergence}
&&\|u-w^H\|_{a,\Omega}=\|u-K\lambda_Hu_H\|_{a,\Omega}=\|K\lambda u-K\lambda_Hu_H\|_{a,\Omega}\nonumber\\
&\lesssim& \|\lambda u-\lambda_Hu_H\|_{0,\Omega}= \mathcal{O}\big(r(h_0)\big)\|u-u_H\|_{a,\Omega}
\end{eqnarray}
Combining (\ref{u_tilde_u_h}), (\ref{u-u_h_w_H}), and (\ref{u_w_H_Convergence}) leads to (\ref{u-u_h_H_Space}).

By the property (\ref{Osc_Bounded}) of oscillation, we have
\begin{eqnarray*}
{\rm osc}(L(R_hw^H-u_h),\mathcal{T}_h) &\lesssim& C_L\|R_hw^H-u_h\|_{a,\Omega},
\end{eqnarray*}
which together  with  (\ref{u_tilde_u_h}) and the fact $\widetilde{u}_h=R_hw^H$ implies
\begin{eqnarray}\label{osc_w_h_w_H}
{\rm osc}(L(R_hw^H-u_h),\mathcal{T}_h) &\lesssim& r(h_0) (\|u-u_H\|_{a,\Omega}+\|u-u_h\|_{a,\Omega}).
\end{eqnarray}
Hence from (\ref{Osc_Bounded}), (\ref{osc_w_h_w_H}), and
\begin{eqnarray}
L(u_h)&=& L(R_hw^h)+L(R_hw^H-u_h),
\end{eqnarray}
we can obtain the desired result (\ref{osc_u_h}).

Now we come to consider the relation (\ref{eta_u_h_R_h_w}).
Using (\ref{u_tilde_u_h}), (\ref{eta_H_1_Bounded}), and the fact $\widetilde{u}_h=R_hw^H$, we obtain
\begin{eqnarray}\label{tilde_eta_w_h_w_H}
\tilde{\eta}_h(R_hw^H-u_h,\Omega) &\lesssim& r(h_0)(\|u-u_H\|_{a,\Omega}+\|u-u_h\|_{a,\Omega}).
\end{eqnarray}
Combining (\ref{eta_sum_inequality}), (\ref{tilde_eta_w_h_w_H}) and the fact
\begin{eqnarray*}
\eta_h(u_h,\Omega) &=& \eta_h(R_hw^H+u_h-R_hw^H,\Omega)\leq \tilde{\eta}_h(R_hw^H,\Omega)+\eta_h(u_h-R_hw^H,\Omega)\nonumber\\ &\leq&\tilde{\eta}_h(R_hw^H,\Omega)+\mathcal{O}(r(h_0))(\|u-u_h\|_{a,\Omega}+\|u-u_H\|_{a,\Omega})
\end{eqnarray*}
leads to the desired result (\ref{eta_u_h_R_h_w}) and the proof is complete.
\end{proof}

Now we are at the position to give the error reduction of {\bf Adaptive Algorithm $C$} for the eigenvalue computations.
\begin{theorem}\label{Convergence_u_h_Theorem}
For the successive eigenfunction approximations $u_H$ and $u_h$ produced by {\bf Adaptive Algorithm $C$},
there exist constants $\gamma>0$, $\alpha_0$, and $\alpha \in (0, 1)$, depending only on the shape
regularity of meshes, $C_a$, $c_a$ and the parameter $\theta$ used by {\bf Adaptive Algorithm $C$}, such that
\begin{eqnarray}\label{Convergence_u_h_u_H_u_H_1}
\|u-u_h\|^2_{a,\Omega}+\gamma \eta_h^2(u_h,\Omega)&\leq&\alpha^2\big(\|u-u_H\|^2_{a,\Omega}
+\gamma\eta^2_H(u_H,\Omega)\big)\nonumber\\
&& + \alpha_0^2 r^2(h_0)\|u-u_{H_{-1}}\|_{a,\Omega},
\end{eqnarray}
provided $h_0\ll 1$.
\end{theorem}
\begin{proof}
Since $w^h=\lambda_hKu_h$ and $w^H=\lambda_HKu_H$, we conclude from Theorem \ref{Adaptive_Convergence_Source_Theorem}
 that there exists constant $\tilde{\gamma}>0$ and $\xi\in (0,1)$ such that
\begin{eqnarray}\label{Convergence_Source_Problem}
\|w^H-R_hw^H\|_{a,\Omega}^2+\tilde{\gamma}\tilde{\eta}_h^2(R_hw^H,\Omega)&\leq &
\xi^2\big(\|w^H-R_Hw^H\|_{a,\Omega}^2\nonumber\\
&&\ \ \ \ +\tilde{\gamma}\tilde{\eta}_H^2(R_Hw^H,\Omega)\big).
\end{eqnarray}
From (\ref{u-u_h_H_Space}) and (\ref{eta_u_h_R_h_w}), there exists a constant $\hat{C}>0$ such that
\begin{eqnarray}\label{Error_Bound_1}
&&\|u-u_h\|_{a,\Omega}^2+\tilde{\gamma}\eta_h^2(u_h,\Omega)\nonumber\\
&\leq&(1+\delta_1)\|w^H-R_hw^H\|_{a,\Omega}^2+\hat{C}(1+\delta_1^{-1})r^2(h_0)(\|u-u_h\|_{a,\Omega}^2
+\|u-u_H\|_{a,\Omega}^2)\nonumber\\
&&+(1+\delta_1)\tilde{\gamma}\tilde{\eta}_h^2(R_hw^H,\Omega)+\hat{C}\tilde{\gamma}(1+\delta_1^{-1})r^2(h_0)
(\|u-u_h\|_{a,\Omega}^2+\|u-u_H\|_{a,\Omega}^2)\nonumber\\
&\leq&(1+\delta_1) \big(\|w^H-R_hw^H\|_{a,\Omega}^2 + \tilde{\gamma}\tilde{\eta}_h^2(R_hw^H,\Omega)\big)\nonumber\\
&& + 
C_4\delta_1^{-1}r^2(h_0)(\|u-u_h\|_{a,\Omega}^2+\|u-u_H\|_{a,\Omega}^2)  ,
\end{eqnarray}
where $C_4$ depends on the constants $\hat{C}$ and $\tilde{\gamma}$ and the Young inequality is used with
 $\delta_1\in (0,1)$ satisfying
\begin{eqnarray*}
(1+\delta_1)\xi^2<1.
\end{eqnarray*}
The similar argument leads to
\begin{eqnarray}\label{Error_Bound_2}
&&\|w^H-R_Hw^H\|_{a,\Omega}^2+\tilde{\gamma}\tilde{\eta}_h^2(R_Hw^H,\Omega)\nonumber\\
&\leq&(1+\delta_2)\|u-u_H\|_{a,\Omega}^2+\hat{C}(1+\delta_2^{-1})r^2(h_0)(\|u-u_H\|_{a,\Omega}^2
+\|u-u_{H_{-1}}\|_{a,\Omega}^2)\nonumber\\
&&+(1+\delta_2)\tilde{\gamma}\eta_H^2(u_H,\Omega)+\hat{C}\tilde{\gamma}(1+\delta_2^{-1})r^2(h_0)
(\|u-u_H\|_{a,\Omega}^2+\|u-u_{H_{-1}}\|_{a,\Omega}^2)\nonumber\\
&=&(1+\delta_2)\big(\|u-u_H\|_{a,\Omega}^2+\tilde{\gamma}\eta_H^2(u_H,\Omega)\big)\nonumber\\
&& + 
C_4\delta_2^{-1}r^2(h_0)
(\|u-u_H\|_{a,\Omega}^2+\|u-u_{H_{-1}}\|_{a,\Omega}^2),
\end{eqnarray}
where $u_{H_{-1}}$ denotes the eigenfunction approximation obtained on the mesh level $\mathcal{T}_{h_{k-1}}$
 before $\mathcal{T}_H$ and $\delta_2\in (0,1)$ satisfies
\begin{eqnarray}\label{Def_delta_1_delta_2}
(1+\delta_1)(1+\delta_2+C_4\delta_2^{-1}r^2(h_0))\xi^2<1.
\end{eqnarray}

Combing (\ref{Convergence_Source_Problem}) and (\ref{Error_Bound_1}) leads to
\begin{eqnarray}\label{COnvergence_Derive_1}
&&\|u-u_h\|_{a,\Omega}^2+\tilde{\gamma}\eta_h^2(u_h,\Omega)\nonumber\\
&\leq&(1+\delta_1)\xi^2\big(\|w^H-R_Hw^H\|_{a,\Omega}^2+\tilde{\gamma}\tilde{\eta}_H^2(R_Hw^H,\Omega)\Big)\nonumber\\
&&+C_4\delta_1^{-1}r^2(h_0)\big(\|u-u_h\|_{a,\Omega}^2+\|u-u_H\|_{a,\Omega}^2\big).
\end{eqnarray}
From (\ref{Error_Bound_2}) and (\ref{COnvergence_Derive_1}), we have
\begin{eqnarray}
\hskip-0.5cm\|u-u_h\|_{a,\Omega}^2+\tilde{\gamma}\eta_h^2(u_h,\Omega)
&\leq&(1+\delta_1)\xi^2\Big((1+ \delta_2)\big(\|u-u_H\|_{a,\Omega}^2+\tilde{\gamma}\eta_H^2(u_H,\Omega)\big)\nonumber\\
&&\hskip-0.5cm+C_4\delta_2^{-1}r^2(h_0)\big(\|u-u_H\|_{a,\Omega}^2+\|u-u_{H_{-1}}\|^2_{a,\Omega}\big)\Big)\nonumber\\
&&\hskip-0.5cm +C_4\delta_1^{-1}r^2(h_0)\Big(\|u-u_h\|_{a,\Omega}^2+\|u-u_H\|_{a,\Omega}^2\Big).
\end{eqnarray}
Consequently,
\begin{eqnarray}
&&\big(1-C_4\delta_1^{-1}r^2(h_0)\big)\|u-u_h\|_{a,\Omega}^2+\tilde{\gamma}\eta_h^2(u_h,\Omega)\nonumber\\
&\leq&\big((1+\delta_1)(1+\delta_2+C_4\delta_2^{-1}r^2(h_0))\xi^2
+C_4\delta_1^{-1}r^2(h_0)\big)\|u-u_H\|_{a,\Omega}^2\nonumber\\
&&\ \ \ \  +(1+\delta_1)(1+\delta_2)\xi^2\tilde{\gamma}\eta_H^2(u_H,\Omega)\nonumber\\
&&\ \ \ \ +C_4(1+\delta_1)\delta_2^{-1}\xi^2r^2(h_0)\|u-u_{H_{-1}}\|_{a,\Omega}^2,
\end{eqnarray}
that is
\begin{eqnarray}
&&\|u-u_h\|^2_{a,\Omega}+\frac{\tilde{\gamma}}{1-C_4\delta_1^{-1}r^2(h_0)}\eta_H^2(u_h,\Omega)\nonumber\\
&\leq& \frac{(1+\delta_1)(1+\delta_2+C_4\delta_2^{-1}r^2(h_0))\xi^2+C_4\delta_1^{-1}r^2(h_0)}{1-C_4\delta_1^{-1}r^2(h_0)}
\|u-u_H\|^2_{a,\Omega}\nonumber\\
&&+\frac{(1+\delta_1)(1+\delta_2)\xi^2\tilde{\gamma}}{1-C_4\delta_1^{-1}r^2(h_0)}\eta^2_H(u_H,\Omega)\nonumber\\
&& + \frac{C_4(1+\delta_1)\delta_2^{-1}\xi^2}{1-C_4\delta_1^{-1}r^2(h_0)}r^2(h_0)\|u-u_{H_{-1}}\|_{a,\Omega}^2.
\end{eqnarray}
Since $h_0\ll 1$ implies $r(h_0)\ll 1$, we have that the constant $\alpha$ defined by
\begin{eqnarray}
 \alpha:=\left( \frac{(1+\delta_1)(1+\delta_2+C_4\delta_2^{-1}r^2(h_0))\xi^2+C_4\delta_1^{-1}r^2(h_0)}{1-C_4\delta_1^{-1}r^2(h_0)} \right)^{1/2}
\end{eqnarray}
satisfying $\alpha\in(0,1)$.  Therefore
\begin{eqnarray}
&&\|u-u_h\|^2_{a,\Omega}+\frac{\tilde{\gamma}}{1-C_4\delta_1^{-1}r^2(h_0)}\eta_H^2(u_h,\Omega)
\leq \alpha^2\Big(\|u-u_H\|^2_{a,\Omega}\nonumber\\
&&+\frac{(1+\delta_1)(1+\delta_2)\xi^2\tilde{\gamma}}{(1+\delta_1)(1+\delta_2+C_4\delta_2^{-1}r^2(h_0))\xi^2
+C_4\delta_1^{-1}r^2(h_0)} \eta^2_H(u_H,\Omega)\Big)\nonumber\\
&&+ \frac{C_4(1+\delta_1)\delta_2^{-1}\xi^2r^2(h_0)}{1-C_4\delta_1^{-1}r^2(h_0)}\|u-u_{H_{-1}}\|_{a,\Omega}^2.
\end{eqnarray}
If we choose
\begin{eqnarray}\label{Def_Gamma}
\gamma:=\frac{\tilde{\gamma}}{1-C_4\delta_1^{-1}r^2(h_0)},
\end{eqnarray}
we arrive at (\ref{Convergence_u_h_u_H_u_H_1}) by using the fact
\begin{eqnarray*}
\frac{(1+\delta_1)(1+\delta_2)\xi^2\tilde{\gamma}}
{(1+\delta_1)(1+\delta_2+C_4\delta_2^{-1}r^2(h_0))\xi^2+C_4\delta_1^{-1}r^2(h_0)} <\gamma
\end{eqnarray*}
and setting
\begin{eqnarray*}
\alpha_0^2 = \frac{C_4(1+\delta_1)\delta_2^{-1}\xi^2}{1-C_4\delta_1^{-1}r^2(h_0)}.
\end{eqnarray*}
Hence the proof is complete.
\end{proof}

Based on Theorem \ref{Convergence_u_h_Theorem}, we can give the following error estimate for {\bf Adaptive Algorithm $C$}.
\begin{theorem}\label{Convergence_u_h_lambda_h_Theorem}
 Suppose $(\lambda,u)\in \mathcal{R}\times H_0^1(\Omega)$ be a simple eigenpair of (\ref{eigenvalue problem}),
 $(\lambda_{h_k},u_{h_k})$ be a sequence of finite element solutions produced by {\bf Adaptive Algorithm C}.
 When $h_0$ is small enough, there exist constants $\beta>0$ and $\bar{\alpha}\in(0,1)$, depending
 on the shape regularity of meshes and the parameter $\theta$, such that for any two consecutive iterates $k$ and $k+1$
\begin{eqnarray}\label{Convergence_Adaptive_Eigenfunct_1}
d_{h_{k+1}}^2\leq \alpha^2d_{h_k}^2+\alpha_0^2 r^2(h_0)\|u-u_{h_{k-1}}\|_{a,\Omega}^2,
\end{eqnarray}
and
\begin{eqnarray}\label{Convergence_Adaptive_Eigenfunct_2}
d_{h_{k+1}}^2+\beta^2 r^2(h_0) d_{h_k}^2\leq \bar{\alpha}^2\big(d_{h_k}^2+\beta^2 r^2(h_0) d_{h_{k-1}}^2\big),
\end{eqnarray}
where $d^2_{h_k}=\|u-u_{h_k}\|_{a,\Omega}^2+\gamma \eta_{h_k}^2(u_{h_k},\Omega)$.
 Then, {\bf Adaptive Algorithm C} converges with a linear rate $\bar{\alpha}$, i.e. the $n$-th iteration solution
 $(\lambda_{h_n},u_{h_n})$ of {\bf Adaptive Algorithm $C$} has the following error estimates
\begin{eqnarray}
d_{h_n}^2+\beta^2 r^2(h_0) d_{h_{n-1}}^2&\leq& C_0\bar{\alpha}^{2n},\label{Error_Adaptive_Eigenfunct}\\
|\lambda_{h_n}-\lambda| &\lesssim& \bar{\alpha}^{2n},\label{Error_Adaptive_Eigenvalue}
\end{eqnarray}
where $C_0=\|u-u_{h_{0}}\|_{a,\Omega}^2+\gamma\eta^2_{h_{0}}(u_{h_{0}},\Omega)$.
\end{theorem}
\begin{proof}
It is obvious that (\ref{Convergence_Adaptive_Eigenfunct_1}) can be derived directly from Theorem \ref{Convergence_u_h_Theorem}.  Now we choose $\bar{\alpha}$ and $\beta$ such that
\begin{eqnarray*}
\bar{\alpha}^2-\beta^2r^2(h_0) &=&\alpha^2,\\
\bar{\alpha}^2\beta^2 &=&\alpha_0^2.
\end{eqnarray*}
This equation leads to
 $$\bar{\alpha}^2=\frac{\alpha^2+\sqrt{\alpha^4+4\alpha_0^2r^2(h_0)}}{2} \ \ {\rm and}\ \
\beta^2=\frac{2\alpha_0^2}{\alpha^2+\sqrt{\alpha^4+4\alpha_0^2r^2(h_0)}}.$$
As we know $\bar{\alpha}<1$ provided that $\alpha<1$ and $h_0$ is small enough. Then (\ref{Convergence_Adaptive_Eigenfunct_2})
can be obtained with the chosen constants $\bar{\alpha}$ and $\beta$.

The estimates (\ref{Error_Adaptive_Eigenfunct}) and (\ref{Error_Adaptive_Eigenvalue}) are
natural results from (\ref{Convergence_Adaptive_Eigenfunct_2}) and (\ref{Expansion_Eigenvalue}) and
the proof is complete.
\end{proof}

\section{Complexity analysis}

Due to Theorems \ref{Adaptive_Convergence_Source_Theorem} and \ref{Convergence_u_h_Theorem},
we are able to analyze the complexity of {\bf Adaptive Algorithm $C$} for eigenvalue problem via
the complexity result of the associated boundary value problems.

In this section, we assume the initial mesh size $h_0$ is small enough such that
\begin{eqnarray}\label{Assumption_h_0}
r(h_0)\|u-u_{h_{k-1}}\|_{a,\Omega}^2 \leq \|u-u_{h_k}\|^2_{a,\Omega}. 
\end{eqnarray}
Then from Theorem \ref{Convergence_u_h_Theorem}, we have the following error reduction property of {\bf Adaptive Algorithm $C$}
\begin{eqnarray}\label{Convergence_u_h_k_u_h_k_1}
\|u-u_{h_{k+1}}\|^2_{a,\Omega}+\gamma\eta^2_{h_{k+1}}(u_{h_{k+1}},\Omega) \leq \tilde{\alpha}^2 \big(\|u-u_{h_k}\|^2_{a,\Omega}+\gamma\eta^2_{h_k}(u_{h_k},\Omega)\big)
\end{eqnarray}
with $\tilde{\alpha}^2=\alpha^2+\alpha_0^2r(h_0)$.

Based on this contraction result, we also give the complexity analysis with the similar way of \cite{CasconKreuzerNochettoSiebert}
and \cite{DaiXuZhou}.
Let $\mathcal{T}_{h_k} (k\geq 0)$ be the sequence of conforming nested partitions generated by {\bf REFINE}
starting from $\mathcal{T}_{h_0}$ with $h_0\ll 1$. We denote $\mathcal{T}_{h_{k,*}}$ a refinement of $\mathcal{T}_{h_k}$
(in general nonconforming), $\mathcal{M}(\mathcal{T}_{h_k})$ the set of elements of $\mathcal{T}_{h_k}$ that were refined
 in $\mathcal{T}_{h_k}$. Let $I_{h_{k+1}}: C(\Omega)\cap H_0^1(\Omega)\rightarrow V_{h_{k+1}}$ satisfy
\begin{eqnarray*}
I_{h_{k+1}}v=v\ \ {\rm on}\  T\not\in \mathcal{M}(\mathcal{T}_{h_k})\ \ \forall v\in V_{h_{k+1}}
\end{eqnarray*}
and set
\begin{equation*}
V_{h_{k,*}} =\left\{
\begin{array}{ll}
V_{h_k}\cup \big((I-I_{h_{k+1}})V_{h_{k+1}}\big), & \ {\rm if}\ \mathcal{T}_{h_{k,*}}\ {\rm is\ nonconforming},\\
V_{h_{k+1}}, &\ {\rm if}\ \mathcal{T}_{h_{k,*}}\ {\rm is\ conforming}.
\end{array}
\right.
\end{equation*}
In our analysis, we also need the following result (see, e.g.,
\cite{DaiXuZhou,CasconKreuzerNochettoSiebert,Nochetto,Stevenson_2007,Stevson_2008}).
\begin{lemma}\label{Number_Up_Bound}
 (Complexity of {\bf REFINE}) Let $\mathcal{T}_{h_k}$ $(k\geq 0)$ be a sequence of
conforming nested partitions generated by {\bf REFINE} starting from $\mathcal{T}_{h_0}$, $\mathcal{M}(T_{h_{k,*}})$ the
set of elements of $\mathcal{T}_{h_k}$ which is marked for refinement and $\mathcal{T}_{h_{k,*}}$ be the partition created
by refinement of elements only in $\mathcal{M}(T_{h_{k,*}})$. There exists a constant $\hat{C}_0$ depending solely
on $\mathcal{T}_{h_0}$  such that
\begin{eqnarray}\label{Refinement_bounded}
\#\mathcal{T}_{h_{k+1}}- \#\mathcal{T}_{h_{0}}\leq \hat{C}_0\sum\limits_{i=0}^k\big(\#T_{h_{i,*}}-\#T_{h_{i}}\big).
\end{eqnarray}
Here and hereafter in this paper, we use $\#\mathcal{T}$ to denote the number of elements in the mesh $\mathcal{T}$.
\end{lemma}

In order to analyze the complexity of {\bf Adaptive Algorithm $C$}, we first review some results related to the analysis of
complexity for the boundary value problem (\ref{Source_Problem_Weak}). For the proofs, please read the papers
\cite{CasconKreuzerNochettoSiebert} and \cite{DaiXuZhou}.

\begin{lemma}(\cite{CasconKreuzerNochettoSiebert})\label{Error_reduction_Refinement_Lemma}
 Let $R_{h_k}u\in V_{h_k}$ and $R_{h_{k,*}}u\in V_{h_{k,*}}$ be the discrete solutions of (\ref{Source_Problem_Weak}) on the meshes
 $\mathcal{T}_{h_k}$ and its refinement $\mathcal{T}_{h_{k,*}}$ with marked element $\mathcal{M}(\mathcal{T}_{h_{k,*}})$.
Then we have
\begin{eqnarray}
 \|R_{h_k}u-R_{h_{k,*}}u\|^2_{a,\Omega}\leq \tilde{C}_1\sum\limits_{T\in\mathcal{M}(\mathcal{T}_{h_{k,*}})}\tilde{\eta}^2_{h_k}(R_{h_k}u,T).
\end{eqnarray}
\end{lemma}

\begin{lemma}(\cite{CasconKreuzerNochettoSiebert,DaiXuZhou})\label{Error_estimate_Lower_Bound_Lemma}
 Under the same assumptions as in Lemma \ref{Error_reduction_Refinement_Lemma} and the energy decrease property
\begin{eqnarray}
&&\|u-R_{h_{k,*}}u\|^2_{a,\Omega}+\tilde{\gamma}_0{\rm osc}(f-L(R_{h_{k,*}}u))^2\nonumber\\
&& \ \ \ \ \ \leq \tilde{\xi}^2_0(\|u-R_{h_k}u\|^2_{a,\Omega}+\tilde{\gamma}_0 {\rm osc}(f-L(R_{h_k}u))^2)
\end{eqnarray}
with $\tilde{\gamma}_0>0$ and $\tilde{\xi}_0^2\in(0,\frac{1}{2})$. Then the set $\mathcal{M}(\mathcal{T}_{h_{k,*}})$ of marked elements
satisfy the D\"{o}rfler property
\begin{eqnarray}
\sum\limits_{T\in\mathcal{M}(\mathcal{T}_{h_{k,*}})}\tilde{\eta}^2_{h_k}(R_{h_k}u,T)\geq \tilde{\theta}\sum\limits_{T\in\mathcal{T}_{h_k}}\tilde{\eta}^2_{h_k}(R_{h_k}u,T),
\end{eqnarray}
where $\tilde\theta=\frac{\tilde{C}_2^2(1-2\tilde{\xi}_0^2)}{\tilde{C}_0(\tilde{C}_1^2+(1+2C_L^2\tilde{C}_1^2)\tilde{\gamma}_0)}$ with  $\tilde{C}_0=\max\{1,\frac{\tilde{C}_3^2}{\tilde{\gamma}_0}\}$.
\end{lemma}

As in the normal analysis of AFEM for boundary value problems,
we introduce a function approximation class as follows
$$\mathcal{A}_{\gamma}^s:=\big\{v\in H_0^1(\Omega):|v|_{s,\gamma}<\infty\big\},$$
 where $\gamma>0$ is a constant and
 $$|v|_{s,\gamma}=\sup\limits_{\varepsilon>0}\varepsilon\inf\limits_{\{\mathcal{T}\subset\mathcal{T}_{h_0}:
\inf(\|v-v_{\mathcal{T}}\|^2_1+(\gamma+1)\text{osc}(L(v_{\mathcal{T}},\mathcal{T}))^2)^{1/2}\leq \varepsilon\}}(\#\mathcal{T}-\#\mathcal{T}_{h_0})^s$$
and $\mathcal{T}\subset \mathcal{T}_{h_0}$ means $\mathcal{T}$ is a refinement of $\mathcal{T}_{h_0}$.
From the definition, for $\gamma>0$, we see that $\mathcal{A}_{\gamma}^s=\mathcal{A}_{1}^s$ and we denote
$\mathcal{A}^s$ as $\mathcal{A}_1^s$, $|v|_s$ as $|v|_{s,\gamma}$ for simplicity. Hence the symbol
$\mathcal{A}^s$ is the class of functions that can be approximated within a
given tolerance $\varepsilon$ by continuous piecewise polynomial functions over a partition $\mathcal{T}$
with the number of degrees of freedom $\#\mathcal{T}-\#\mathcal{T}_{h_0}\lesssim  \varepsilon^{-1/s}|v|_s^{1/s}$.

In order to give the proof of optimal complexity of {\bf Adaptive Algorithm $C$} for the eigenvalue problem
(\ref{eigenproblem}), we should give some preparations. Associated with the eigenpair approximation $(\lambda_{h_k},u_{h_k})$
of (\ref{diseig}) in the mesh $\mathcal{T}_{h_k}$, we define $w^{h_k} = K(\lambda_{h_k} u_{h_k})$ as in (\ref{w_h_k_defintion}).

Using the assumption (\ref{Assumption_h_0}) and the similar procedure as in the proof of Theorem
\ref{Convergence_u_h_Theorem} when (\ref{eta_u_h_R_h_w}) is replaced by (\ref{osc_u_h}), we have
\begin{lemma}\label{Contraction_R_h_k_*_Lemma}
Let $(\lambda_{h_k},u_{h_k})\in\mathcal{R}\times V_{h_k}$ and $(\lambda_{h_{k,*}},u_{h_{k,*}})\in\mathcal{R}\times V_{h_{k,*}}$ be the discrete
solutions of (\ref{eigenproblem}) produced by {\bf Adaptive Algorithm $C$} over a conforming mesh $\mathcal{T}_{h_k}$ and its (nonconforming) refinement $\mathcal{T}_{h_{k,*}}$ with marked
element $\mathcal{M}(\mathcal{T}_{h_{k,*}})$. Supposing they satisfy the following property
\begin{eqnarray}
&&\|u-u_{h_{k,*}}\|_{a,\Omega}^2 +\gamma_* {\rm osc}(L(u_{h_{k,*}}),\mathcal{T}_{h_{k,*}})^2 \nonumber\\
&&\ \ \ \ \leq\beta_*^2(\|u-u_{h_k}\|_{a,\Omega}^2+\gamma_*{\rm osc}(L(u_{h_k}),\mathcal{T}_{h_k})^2),
\end{eqnarray}
where $\gamma_*>0$, $\beta_*>0$ are some constants. Then the associated boundary value problem
approximations $R_{h_k}w^{h_k}$ and $R_{h_{k,*}}w^{h_k}$ of $w^{h_k}$ have the following contraction property
\begin{eqnarray}\label{Contraction_R_h_k_*}
&&\|w^{h_k}-R_{h_{k,*}}w^{h_k}\|_{a,\Omega}^2 + \gamma_*{\rm osc}(L(R_{h_{k,*}}w^{h_k}),\mathcal{T}_{h_{k,*}})^2\nonumber\\
&&\ \ \ \ \leq \tilde{\beta}_*^2(\|w^{h_k}-R_{h_k}w^{h_k}\|_{a,\Omega}^2+\gamma_*{\rm osc}(L(u_{h_k}),\mathcal{T}_{h_k})^2)
\end{eqnarray}
with
\begin{eqnarray}\label{Def_tilde_Beta_*}
 \tilde{\beta}_*:=\left( \frac{(1+\delta_1+C_4\delta_1^{-1}r^2(h_0))\beta_*^2+C_4\delta_1^{-1}r^2(h_0)}{1-C_4r(h_0)(1+r(h_0))} \right)^{1/2},
\end{eqnarray}
where the constant $C_4$ depends on $\delta_1\in (0,1)$and $\gamma_*$
as in the proof of Theorem \ref{Convergence_u_h_Theorem}.
\end{lemma}
\begin{proof}
From (\ref{Osc_Bounded}), (\ref{u-u_h_H_Space}), and (\ref{osc_u_h}), there exists a constant $\hat{C}>0$ such that
\begin{eqnarray}\label{Lemma_5_4_inequality_1}
&&\|w^{h_k}-R_{h_{k,*}}w^{h_k}\|_{a,\Omega}^2+\gamma_*{\rm osc}(L(R_{h_{k,*}}w^{h_k}),\mathcal{T}_{h_{k,*}})^2 \nonumber\\
&\leq&(1+\delta_1)\|u-u_{h_{k,*}}\|_{a,\Omega}^2+\hat{C}(1+\delta_1^{-1})r^2(h_0)\big(\|u-u_{h_{k,*}}
\|_{a,\Omega}^2+\|u-u_{h_k}\|_{a,\Omega}^2\big)\nonumber\\
&& +(1+\delta_1)\gamma_*{\rm osc}(L(u_{h_{k,*}}),\mathcal{T}_{h_{k,*}})^2+\hat{C}(1+\delta_1^{-1})r^2(h_0)(\|u-u_{h_{k,*}}
\|_{a,\Omega}^2\nonumber\\
&&\ \ \ \ \ \ \ \  +\|u-u_{h_k}\|_{a,\Omega}^2)\nonumber\\
&\leq& (1+\delta_1)\big(\|u-u_{h_{k,*}}\|_{a,\Omega}^2+\gamma_*{\rm osc}(L(u_{h_{k,*}}),\mathcal{T}_{h_k})^2\big)\nonumber\\
&&+C_4\delta_1^{-1}r^2(h_0)\big(\|u-u_{h_{k,*}}
\|_{a,\Omega}^2+\|u-u_{h_k}\|_{a,\Omega}^2\big)\nonumber\\
&\leq& (1+\delta_1)\beta_*^2\big(\|u-u_{h_k}\|_{a,\Omega}^2+\gamma_*{\rm osc}(L(u_{h_k}),\mathcal{T}_{h_k})^2\big)\nonumber\\
&&+C_4\delta_1^{-1}r^2(h_0)\beta_*^2\big(\|u-u_{h_k}\|_{a,\Omega}^2+\gamma_*{\rm osc}(L(u_{h_k}),\mathcal{T}_{h_k})^2\big)\nonumber\\
&&+C_4\delta_1^{-1}r^2(h_0)\|u-u_{h_k}\|_{a,\Omega}^2\nonumber\\
&\leq&\big((1+\delta_1+C_4\delta_1^{-1}r^2(h_0))\beta_*^2+C_4\delta_1^{-1}r^2(h_0)\big)
\Big(\|u-u_{h_k}\|_{a,\Omega}^2\nonumber\\
&&\ \ \ \ \ \ +\gamma_*{\rm osc}(L(u_{h_k}),\mathcal{T}_{h_k})^2\Big),
\end{eqnarray}
where $C_4$ depends on the constants $\hat{C}$ and $\gamma_*$ and the Young inequality is used.

Similarly from (\ref{Osc_Bounded}), (\ref{u-u_h_H_Space}), (\ref{osc_u_h}), and (\ref{Assumption_h_0}), we have
\begin{eqnarray*}
&&\|u-u_{h_k}\|_{a,\Omega}^2 +\gamma_*{\rm osc}(L(u_{h_k}),\mathcal{T}_{h_k})^2\nonumber\\
&\leq&\|w^{h_k}-R_{h_k}w^{h_k}\|_{a,\Omega}^2+\gamma_*{\rm osc}(L(R_{h_k}w^{h_k}),\mathcal{T}_{h_k})^2\nonumber\\
&&\ \ \ \ + C_4r^2(h_0)\big(\|u-u_{h_k}\|_{a,\Omega}^2+\|u-u_{h_{k-1}}\|_{a,\Omega}^2\big)\nonumber\\
&\leq&\|w^{h_k}-R_{h_k}w^{h_k}\|_{a,\Omega}^2+\gamma_*{\rm osc}(L(R_{h_k}w^{h_k}),\mathcal{T}_{h_k})^2\nonumber\\
&&\ \ \ \ + C_4r(h_0)(1+r(h_0))\|u-u_{h_k}\|_{a,\Omega}^2.
\end{eqnarray*}
Then the following inequality holds
\begin{eqnarray}\label{Lemma_5_4_inequality_2}
&&\|u-u_{h_k}\|_{a,\Omega}^2 +\gamma_*{\rm osc}(L(u_{h_k}),\mathcal{T}_{h_k})^2\nonumber\\
&\leq& \frac{1}{1-C_4r(h_0)(1+r(h_0))}\big(\|w^{h_k}-R_{h_k}w^{h_k}\|_{a,\Omega}^2\nonumber\\
&&\ \ \ \ +\gamma_*{\rm osc}(L(R_{h_k}w^{h_k}),\mathcal{T}_{h_k})^2\big).
\end{eqnarray}
Combining (\ref{Lemma_5_4_inequality_1}) and (\ref{Lemma_5_4_inequality_2}) leads to
\begin{eqnarray}
&&\|w^{h_k}-R_{h_{k,*}}w^{h_k}\|_{a,\Omega}^2+\gamma_*{\rm osc}(L(R_{h_{k,*}}w^{h_k}))^2 \nonumber\\
&\leq&\frac{(1+\delta_1+C_4\delta_1^{-1}r^2(h_0))\beta_*^2+C_4\delta_1^{-1}r^2(h_0)}{1-C_4r(h_0)(1+r(h_0))}
\big(\|w^{h_k}-R_{h_k}w^{h_k}\|_{a,\Omega}^2\nonumber\\
&&\ \ \ \ +\gamma_*{\rm osc}(L(R_{h_k}w^{h_k}),\mathcal{T}_{h_k})^2\big).
\end{eqnarray}
This is the desired result (\ref{Contraction_R_h_k_*}) and the proof is complete.
\end{proof}

We present the following statement which is a direct consequence of Lemmas \ref{Error_estimate_Lower_Bound_Lemma} and \ref{Contraction_R_h_k_*_Lemma}.
\begin{corollary}\label{Refine_Element_Estimate_Corollary}
Let $(\lambda_{h_k},u_{h_k})\in\mathcal{R}\times V_{h_k}$ and $(\lambda_{h_{k,*}},u_{h_{k,*}})\in\mathcal{R}\times V_{h_{k,*}}$  be as
in Lemma \ref{Contraction_R_h_k_*_Lemma}. Suppose they satisfy the following decrease property
\begin{eqnarray*}
\|u-u_{h_{k,*}}\|_{a,\Omega}^2 +\gamma_* {\rm osc}(L(u_{h_{k,*}}),\mathcal{T}_{h_{k,*}})^2
\leq\beta_*^2\big(\|u-u_{h_k}\|_{a,\Omega}^2+\gamma_*{\rm osc}(L(u_{h_k}),\mathcal{T}_{h_k})^2\big),
\end{eqnarray*}
where the constants $\gamma_*>0$ and $\beta_*^2\in (0, 1/2)$. Then the set $\mathcal{M}(\mathcal{T}_{h_{k,*}})$ of marked elements
satisfies the following inequality
\begin{eqnarray}
\sum\limits_{T\in\mathcal{M}(\mathcal{T}_{h_{k,*}})}\eta^2_{h_k}(u_{h_k},T)\geq \hat{\theta}\sum\limits_{T\in\mathcal{T}_{h_k}}\eta^2_{h_k}(u_{h_k},T),
\end{eqnarray}
where the constant $\hat{\theta}=\frac{\tilde{C}_2^2(1-2\tilde{\beta}_*^2)}{\tilde{C}_0(\tilde{C}_1^2+(1+2C_L^2\tilde{C}_1^2)\gamma_*)}$
and $\tilde{C}_0=\max\{1,\frac{\tilde{C}_3^2}{\gamma_*}\}$ with $\tilde{\beta}_*$ and $\gamma_*$
 which are the same as in (\ref{Contraction_R_h_k_*}) and (\ref{Def_tilde_Beta_*}) with $\delta_1$
 being chosen such that $\tilde{\beta}_*^2\in (0, 1/2)$.
\end{corollary}

\begin{lemma}(Upper Bound of DOF).
Let $u\in \mathcal{A}^s$ and $\mathcal{T}_{h_k}$ be a conforming partition from $\mathcal{T}_{h_0}$.
Let $\mathcal{T}_{h_{k,*}}$ be a mesh created from $\mathcal{T}_{h_k}$ by marking the set
$\mathcal{M}(\mathcal{T}_{h_{k,*}})$ according to {\bf Marking Strategy $E$} with $\theta\in (0, \frac{C_2^2\gamma}{C_3^2(C_1^2+(1+2C_L^2C_1^2)\gamma)})$. Then we have
\begin{eqnarray}\label{Number_T_h_estimate}
\#\mathcal{T}_{h_{k,*}}- \#\mathcal{T}_{h_{0}}\leq C\big(\|u-u_{h_k}\|_{a,\Omega}^2 +\gamma {\rm osc}(L(u_{h_n}),\mathcal{T}_{h_n})^2\big)^{-1/(2s)}|u|_s^{1/s},
\end{eqnarray}
where the constant $C$ depends on the discrepancy between $\theta$ and $\frac{C_2^2\gamma}{C_3^2(C_1^2+(1+2C_L^2C_1^2)\gamma)}$.
\end{lemma}
\begin{proof}
We choose $\beta$, $\beta_1\in (0, 1)$ such that $\beta_1\in (0, \beta)$ and
\begin{eqnarray*}
\theta < \frac{C_2^2\gamma}{C_3^2\Big(C_1^2+(1+2C_L^2C_1^2)\gamma\Big)}(1-\beta^2).
\end{eqnarray*}
Let
\begin{eqnarray*}
\varepsilon = \frac{1}{\sqrt{2}}\beta_1\big(\|u-u_{h_k}\|_{a,\Omega}^2+\gamma{\rm osc}(L(u_{h_k}),\mathcal{T}_{h_k})^2\big)^{1/2}.
\end{eqnarray*}
Let $\delta_1\in (0,1)$ and $\delta_2\in (0,1)$ be the constants such that (\ref{Def_delta_1_delta_2}) and
\begin{eqnarray}\label{Condition_beta_1_delta_1_delta_2}
(1+\delta_1)^2(1+\delta_2)^2\beta_1^2 \leq \beta^2,
\end{eqnarray}
which implies
\begin{eqnarray}\label{Condition_beta_1}
(1+\delta_1)(1+\delta_2)\beta_1^2 <1.
\end{eqnarray}
Let $\mathcal{T}_{h_{\varepsilon}}$ be a refinement of $\mathcal{T}_{h_0}$ with minimum degrees of freedom satisfying
\begin{eqnarray}\label{Def_u_h_varepsilon}
\|u-u_{h_{\varepsilon}}\|_{a,\Omega}^2+(\gamma+1){\rm osc}(L(u_{h_{\varepsilon}}),\mathcal{T}_{h_{\varepsilon}})^2\leq \varepsilon^2,
\end{eqnarray}
where $u_{h_{\varepsilon}}$ denotes the solution of eigenvalue problem (\ref{diseig}) over the mesh $\mathcal{T}_{h_{\varepsilon}}$.
By the definition of $\mathcal{A}^s$, we can get that
\begin{eqnarray}
\hskip-0.5cm\#\mathcal{T}_{h_{\varepsilon}}-\#\mathcal{T}_{h_0}\leq \left(\frac{1}{\sqrt{2}}\beta_1\right)^{-1/s}\big(\|u-u_{h_k}\|_{a,\Omega}^2+\gamma{\rm osc} (L(u_{h_k}),\mathcal{T}_{h_k})^2\big)^{-1/(2s)}|u|_s^{1/s}.
\end{eqnarray}

Let $\mathcal{T}_{h_{k,+}}$ be the smallest (nonconforming) common refinement of $\mathcal{T}_{h_k}$ and $\mathcal{T}_{h_{\varepsilon}}$.
Since both $\mathcal{T}_{h_k}$ and $\mathcal{T}_{h_{\varepsilon}}$ are refinements of $\mathcal{T}_{h_0}$, the number of elements in
$\mathcal{T}_{h_{k,+}}$ that are not in $\mathcal{T}_{h_k}$ is less than the number of elements that must be added to go from $\mathcal{T}_{h_0}$
 to $\mathcal{T}_{h_{\varepsilon}}$,
namely,
\begin{eqnarray}
\#\mathcal{T}_{h_{k,+}}-\#\mathcal{T}_{h_k} \leq \#\mathcal{T}_{h_{\varepsilon}}- \#\mathcal{T}_{h_0}.
\end{eqnarray}
Let $w^{h_{\varepsilon}}=K(\lambda_{h_{\varepsilon}}u_{h_{\varepsilon}})$. From definition, we can easily get
\begin{eqnarray}
{\rm osc}(L(R_{h_{k,+}}w^{h_{\varepsilon}}),\mathcal{T}_{h_{k,+}})
&\leq& {\rm osc}(L(R_{h_{\varepsilon}}w^{h_{\varepsilon}}),\mathcal{T}_{h_{k,+}})\nonumber\\
&& +{\rm osc}(L(R_{h_{k,+}}w^{h_{\varepsilon}}-R_{h_{\varepsilon}}w^{h_{\varepsilon}}),\mathcal{T}_{h_{k,+}})\nonumber\\
&\leq&  {\rm osc}(L(R_{h_{\varepsilon}}w^{h_{\varepsilon}}),\mathcal{T}_{h_{k,+}}) \nonumber\\ && +C_L\|R_{h_{k,+}}w^{h_{\varepsilon}}-R_{h_{\varepsilon}}w^{h_{\varepsilon}}\|_{a,\Omega},
\end{eqnarray}
where Lemma \ref{Osc_Bounded_Lemma} is used. Then by the Young inequality, we have
\begin{eqnarray}
{\rm osc}(L(R_{h_{k,+}}w^{h_{\varepsilon}}),\mathcal{T}_{h_{k,+}})^2
&\leq& 2{\rm osc}(L(R_{h_{\varepsilon}}w^{h_{\varepsilon}}),\mathcal{T}_{h_{k,+}})^2 \nonumber\\ && +2C_L^2\|R_{h_{k,+}}w^{h_{\varepsilon}}-R_{h_{\varepsilon}}w^{h_{\varepsilon}}\|_{a,\Omega}^2.
\end{eqnarray}

Since $\mathcal{T}_{h_{k,+}}$ is a refinement of $\mathcal{T}_{h_{\varepsilon}}$, $L^2$-projection error are monotone
and the following orthogonality
\begin{eqnarray}
\|w^{h_{\varepsilon}}-R_{h_{k,+}}w^{h_{\varepsilon}}\|_{a,\Omega}^2 =
\|w^{h_{\varepsilon}}-R_{h_{\varepsilon}}w^{h_{\varepsilon}}\|_{a,\Omega}^2 -
\|R_{h_{k,+}}w^{h_{\varepsilon}}-R_{h_{\varepsilon}}w^{h_{\varepsilon}}\|_{a,\Omega}^2
\end{eqnarray}
is valid, we arrive at
\begin{eqnarray}
&&\|w^{h_{\varepsilon}}-R_{h_{k,+}}w^{h_{\varepsilon}}\|_{a,\Omega}^2 +\frac{1}{2C_L^2}{\rm osc}(L(R_{h_{k,+}}w^{h_{\varepsilon}}),\mathcal{T}_{h_{k,+}})^2\nonumber\\
&& \leq \|w^{h_{\varepsilon}}-R_{h_{\varepsilon}}w^{h_{\varepsilon}}\|_{a,\Omega}^2 +\frac{1}{C_L^2}{\rm osc}(L(R_{h_{\varepsilon}}w^{h_{\varepsilon}}),\mathcal{T}_{h_{\varepsilon}})^2.
\end{eqnarray}
Note that (\ref{Def_tilde_Gamma_source}) implies $\tilde{\gamma}\leq \frac{1}{2C_L^2}$ and we obtain that
\begin{eqnarray}
&&\|w^{h_{\varepsilon}}-R_{h_{k,+}}w^{h_{\varepsilon}}\|_{a,\Omega}^2 +\tilde{\gamma}{\rm osc}(L(R_{h_{k,+}}w^{h_{\varepsilon}}),\mathcal{T}_{h_{k,+}})^2\nonumber\\
&& \leq \|w^{h_{\varepsilon}}-R_{h_{\varepsilon}}w^{h_{\varepsilon}}\|_{a,\Omega}^2 +\frac{1}{C_L^2}{\rm osc}(L(R_{h_{\varepsilon}}w^{h_{\varepsilon}}),\mathcal{T}_{h_{\varepsilon}})^2\nonumber\\
&&\leq \|w^{h_{\varepsilon}}-R_{h_{\varepsilon}}w^{h_{\varepsilon}}\|_{a,\Omega}^2 + (\tilde{\gamma}+\sigma){\rm osc}(L(R_{h_{\varepsilon}}w^{h_{\varepsilon}}),\mathcal{T}_{h_{\varepsilon}})^2
\end{eqnarray}
with $\sigma=\frac{1}{C_L^2}-\tilde{\gamma}\in (0,1)$. Applying the similar argument in the proof of
Theorem \ref{Convergence_u_h_Theorem} when (\ref{eta_u_h_R_h_w}) is replaced by (\ref{osc_u_h}),
 we then obtain
\begin{eqnarray}\label{Def_u_h_k+_u_h_varepsilon}
&&\|u-u_{h_{k,+}}\|_{a,\Omega}^2 + \gamma{\rm osc}(L(u_{h_{k,+}}),\mathcal{T}_{h_{k,+}})\nonumber\\
&& \leq \beta_0^2\Big(\|u-u_{h_{\varepsilon}}\|_{a,\Omega}^2 + (\gamma+\sigma){\rm osc}(L(u_{h_{\varepsilon}}),\mathcal{T}_{h_{\varepsilon}})^2\Big)\nonumber\\
&&\leq \beta_0^2\Big(\|u-u_{h_{\varepsilon}}\|_{a,\Omega}^2 +
(\gamma+1){\rm osc}(L(u_{h_{\varepsilon}}),\mathcal{T}_{h_{\varepsilon}})^2\Big),
\end{eqnarray}
where
\begin{eqnarray*}
\beta_0=\left(\frac{(1+\delta_1)\big((1+\delta_2)+C_5\delta_2^{-1}r^2(h_0)\big)
+C_5\delta_1^{-1}r^2(h_0)}{1-C_5\delta_1^{-1}r^2(h_0)}\right)^{1/2}
\end{eqnarray*}
and
\begin{eqnarray*}
\gamma = \frac{\tilde{\gamma}}{1-C_5\delta_1^{-1}r^2(h_0)}.
\end{eqnarray*}
with $C_5$ the constant depending on $C_L$ similar to $C_4$ in the proof of Theorem \ref{Convergence_u_h_Theorem}.
Combing (\ref{Def_u_h_varepsilon}) and (\ref{Def_u_h_k+_u_h_varepsilon}) leads to
\begin{eqnarray*}
\|u-u_{h_{k,+}}\|_{a,\Omega}^2 + \gamma{\rm osc}(L(u_{h_{k,+}}),\mathcal{T}_{h_{k,+}})
\leq \check{\beta}\Big(\|u-u_{h_k}\|_{a,\Omega}^2+\gamma{\rm osc}(L(u_{h_k}),\mathcal{T}_{h_k})^2\Big)
\end{eqnarray*}
with $\check{\beta}=\frac{1}{\sqrt{2}}\beta_0\beta_1$.

It is seen from $h_0\ll 1$ and (\ref{Condition_beta_1}) that $\check{\beta}^2\in (0,\frac{1}{2})$.
Thus by Corollary \ref{Refine_Element_Estimate_Corollary} we have that $\mathcal{T}_{h_{k,+}}$ satisfies
\begin{eqnarray*}
\sum_{T\in\mathcal{M}(\mathcal{T}_{h_{k,+}})}\eta_{h_k}(T)^2 \geq \check{\theta}\sum_{T\in\mathcal{T}_{h_k}}\eta_{h_k}(T)^2,
\end{eqnarray*}
where $\check{\theta}=\frac{\tilde{C}_2^2(1-2\hat{\beta}^2)}{\tilde{C}_0(\tilde{C}_1^2+(1+2C_L^2\tilde{C}_1^2)\hat{\gamma})}$,
$\hat{\gamma}=\frac{\gamma}{1-C_5\delta_1^{-1}r^2(h_0)}$, $\tilde{C}_0=\max\{1,\frac{\tilde{C}_3^2}{\hat{\gamma}}\}$ and
\begin{eqnarray*}
\hat{\beta}=\left( \frac{(1+\delta_1+C_5\delta_1^{-1}r^2(h_0))\check{\beta}^2
+C_5\delta_1^{-1}r^2(h_0)}{1-C_5r(h_0)(1+r(h_0))} \right)^{1/2}.
\end{eqnarray*}
From the definition of $\gamma$ (see (\ref{Def_Gamma})), 
we obtain that $\hat{\gamma}<1$. On the other hand we have
$\tilde{C}_3 > 1 $ and hence $\tilde{C}_0 = \frac{\tilde{C}_3^2}{\hat{\gamma}}$.
 Consequently, we can write $\check{\theta}$ as $\check{\theta}=\frac{\tilde{C}_2^2(1-2\hat{\beta}^2)}{\tilde{C}_3^2(\frac{\tilde{C}_1^2}{\hat{\gamma}}+(1+2C_L^2\tilde{C}_1^2))}$.

Since $h_0\ll 1$ and (\ref{Condition_beta_1_delta_1_delta_2}), we obtain that $\hat{\gamma}>\gamma$ and $\hat{\beta}\in (0, \frac{1}{\sqrt{2}}\beta)$. Using (\ref{Setting_C_1}) and (\ref{Setting_C_2_C_3}), we get
\begin{eqnarray}
\check{\theta} &=& \frac{\tilde{C}_2^2(1-2\hat{\beta}^2)}{\tilde{C}_3^2(\frac{\tilde{C}_1^2}{\hat{\gamma}}+(1+2C_L^2\tilde{C}_1^2))}
\geq \frac{\tilde{C}_2^2}{\tilde{C}_3^2\Big(\frac{\tilde{C}_1^2}{\hat{\gamma}}+(1+2C_L^2\tilde{C}_1^2)\Big)}(1-\beta^2)\nonumber\\
&=& \frac{(2+\tilde{C}r^2(h_0))C_2^2}{(2+\tilde{C}r^2(h_0))C_3^2\Big(\frac{\tilde{C}_1^2}{\hat{\gamma}}
+(1+2C_L^2\tilde{C}_1^2)\Big)}(1-\beta^2)\nonumber\\
&\geq& \frac{C_2^2}{C_3^2\Big(\frac{C_1^2}{\gamma}+(1+2C_L^2C_1^2)\Big)}(1-\beta^2)\nonumber\\
&=&\frac{C_2^2\gamma}{C_3^2\Big(C_1^2+(1+2C_L^2C_1^2)\gamma\Big)}(1-\beta^2) > \theta.
\end{eqnarray}

Note that {\bf Marking Strategy $E$} selects a minimum set $\mathcal{M}(\mathcal{T}_{h_{k,*}})$ satisfying
\begin{eqnarray*}
\sum_{T\in\mathcal{M}(\mathcal{T}_{h_{k,*}})}\eta_{h_k}(T)^2 \geq \theta\sum_{T\in\mathcal{T}_{h_k}}\eta_{h_k}(T)^2,
\end{eqnarray*}
which implies that the nonconforming partition $\mathcal{T}_{h_{k,*}}$ satisfies
\begin{eqnarray}
\#\mathcal{T}_{h_{k,*}}-\#\mathcal{T}_{h_k} & \leq & \#\mathcal{T}_{h_{k,+}}
-\#\mathcal{T}_{h_k}  \leq \# \mathcal{T}_{h_{\varepsilon}} -\#\mathcal{T}_{h_0}\nonumber\\
&&\hskip-2.5cm\leq\left(\frac{1}{\sqrt{2}}\beta_1\right)^{-1/s}\big(\|u-u_{h_k}\|_{a,\Omega}^2+\gamma {\rm osc} (L(u_{h_k}),\mathcal{T}_{h_k})^2\big)^{-1/(2s)}|u|_s^{1/s}.
\end{eqnarray}
This is the desired estimate (\ref{Number_T_h_estimate}) with an explicit dependence on the discrepancy between
$\theta$ and $\frac{C_2^2\gamma}{C_3^2(C_1^2+(1+2C_L^2C_1^2)\gamma)}$ via $\beta_1$.
\end{proof}

We are now in the position to prove the optimal complexity of {\bf Adaptive Algorithm $C$} which is stated in the following theorem.
Please refer the papers \cite{CasconKreuzerNochettoSiebert} and \cite{DaiXuZhou} for the proof.
\begin{theorem}(\cite{CasconKreuzerNochettoSiebert,DaiXuZhou})
 Let $(\lambda, u)\in \mathcal{R}\times(H_0^1(\Omega)\cap \mathcal{A}^s)$ be some simple eigenpair of (\ref{eigenproblem})
  and  $\{(\lambda_{h_k}, u_{h_k})\}_{k\in\mathbb{N}_0}$ be the sequence of finite element approximations corresponding
to the sequence of pairs $\{\mathcal{R}\times V_{h_k}\}_{k\in\mathbb{N}_0}$ produced by {\bf Adaptive Algorithm $C$} satisfying
(\ref{Error_u_u_h_k_-1}). Then under the assumption (\ref{Assumption_h_0}), the $n$-th iterate solution
$(\lambda_{h_n}, u_{h_n})$ of {\bf Adaptive Algorithm $C$} satisfies the optimal bounds
\begin{eqnarray}
\|u-u_{h_k}\|_{a,\Omega}^2 +\gamma {\rm osc}(L(u_{h_n}),\mathcal{T}_{h_n})^2 &\lesssim & (\#\mathcal{T}_{h_n}-\#\mathcal{T}_{h_0})^{-2s},\label{Error_Adaptive_eigfunction_n}\\
|\lambda_{h_n}-\lambda|&\lesssim & (\#\mathcal{T}_{h_n}-\#\mathcal{T}_{h_0})^{-2s},\label{Error_Adaptive_Eigenvalue_n}
\end{eqnarray}
where the hidden constant depends on the exact eigenpair $(\lambda, u)$ and the discrepancy between $\theta$ and $\frac{C_2^2\gamma}{C_3^2(C_1^2+(1+2C_L^2C_1^2)\gamma)}$.
\end{theorem}
\begin{proof}
We give the proof in the same way in \cite{DaiXuZhou}. From (\ref{Lower_Bound_Eigenfunct}) and (\ref{Convergence_u_h_k_u_h_k_1}), we have
\begin{eqnarray}\label{Equivalence_eta_osc}
\|u-u_{h_k}\|_{a,\Omega}^2+\gamma \eta_{h_k}^2(u_{h_k},\Omega)\leq
 \check{C}\Big(\|u-u_{h_k}\|_{a,\Omega}^2+\gamma {\rm osc}(L(u_{h_k}),\mathcal{T}_{h_k})^2\big),
\end{eqnarray}
where $\check{C}$ depends on $C_2$, $C_3$, $\gamma$, and $r(h_0)$.

Combining (\ref{Number_T_h_estimate}) and (\ref{Equivalence_eta_osc}) leads to
\begin{eqnarray}\label{Number_T_h_estimate_eta}
\#\mathcal{T}_{h_{k,*}}-\#\mathcal{T}_{h_k}\leq C\check{C}^{\frac{1}{2s}}\Big(\|u-u_{h_k}\|_{a,\Omega}^2+\gamma\eta_{h_k}(u_{h_k},\Omega)\Big)^{-1/(2s)}|u|^{1/s}.
\end{eqnarray}
From (\ref{Convergence_u_h_k_u_h_k_1}) we have
\begin{eqnarray}\label{Relation_k_n}
&&\Big(\|u-u_{h_k}\|_{a,\Omega}^2+\gamma\eta_{h_k}^2(u_{h_k},\Omega)\Big)^{-1/(2s)}\nonumber\\
&&\hskip2cm \leq \alpha^{(n-k)/s}
\Big(\|u-u_{h_n}\|_{a,\Omega}^2+\gamma\eta_{h_n}^2(u_{h_n},\Omega)\Big)^{-1/(2s)}.
\end{eqnarray}
Employing Lemma \ref{Number_Up_Bound}, (\ref{Number_T_h_estimate}), and (\ref{Relation_k_n}), we can
obtain
\begin{eqnarray*}
\#\mathcal{T}_{h_n}-\#\mathcal{T}_{h_0} &\lesssim& \sum_{k=0}^{n-1}\big(\#\mathcal{T}_{h_{k,*}}-\mathcal{T}_{h_k}\big)\nonumber\\
&\lesssim& |u|_s^{1/s}\sum_{k=0}^{n-1}\Big(\|u-u_{h_k}\|_{a,\Omega}^2+\gamma\eta_{h_k}^2(u_{h_k},\Omega)\Big)^{-1/(2s)}\nonumber\\
&\lesssim& \Big(\|u-u_{h_n}\|_{a,\Omega}^2+\gamma\eta_{h_n}^2(u_{h_n},\Omega)\Big)^{-1/(2s)}|u|_s^{1/s}
\sum_{k=1}^n\alpha^{\frac{k}{s}}.
\end{eqnarray*}
Combining the fact $\alpha<1$ leads to
\begin{eqnarray*}
\#\mathcal{T}_{h_n}-\#\mathcal{T}_{h_0} &\lesssim& \Big(\|u-u_{h_n}\|_{a,\Omega}^2+\gamma\eta_{h_n}^2(u_{h_n},\Omega)\Big)^{-1/(2s)}|u|_s^{1/s}.
\end{eqnarray*}
Sice ${\rm osc}(L(u_h),\mathcal{T}_{h_n})\leq \eta_{h_n}(u_{h_n},\Omega)$, we have
\begin{eqnarray*}
\#\mathcal{T}_{h_n}-\#\mathcal{T}_{h_0} &\lesssim& \Big(\|u-u_{h_n}\|_{a,\Omega}^2+\gamma{\rm osc}(u_{h_n},\mathcal{T}_{h_n})^2\Big)^{-1/(2s)}|u|_s^{1/s}.
\end{eqnarray*}
This is the desired result (\ref{Error_Adaptive_eigfunction_n}) and (\ref{Error_Adaptive_Eigenvalue_n})
 can be derived from (\ref{Error_Adaptive_eigfunction_n}) and Lemma \ref{Expansion_Eigenvalue_Lemma}.
 Then the proof is complete.
\end{proof}

\section{Numerical experiments}
In this section, we present sone numerical examples of {\bf Adaptive Algorithm $C$} for the second order elliptic
eigenvalue problems by the linear finite element method.

{\bf Example 1.} In this example, we consider the following eigenvalue problem (see \cite{Greiner})
\begin{equation}\label{eigenproblem_Exam_1}
\left\{
\begin{array}{rcl}
-\frac{1}{2}\Delta u +\frac{1}{2}|x|^2u&=&\lambda u\ \ \ {\rm in}\ \Omega,\\
u&=&0\ \ \ \ \ {\rm on}\ \partial\Omega,\\
\|u\|_{0,\Omega}&=&1,
\end{array}
\right.
\end{equation}
where $\Omega= \mathcal{R}^2$ and $|x|=\sqrt{x_1^2+x_2^2}$.
The first eigenvalue of (\ref{eigenproblem_Exam_1}) is $\lambda=1$ and the associated eigenfunction
is $u=\kappa e^{-|x|^2/2}$ with any nonzero constant $\kappa$.
In our computation, we set $\Omega=(-5, 5)\times (-5, 5)$.

First, we investigate the numerical results for the first eigenvalue approximations.
We give the numerical results for the eigenpair approximation by {\bf Adaptive Algorithm $C$} with the parameter $\theta = 0.4$.
Figure \ref{Mesh_AFEM_Exam_1} shows the initial triangulation and the triangulation after $14$
adaptive iterations. Figure \ref{Convergence_AFEM_Exam_1_First}
 gives the corresponding numerical results for the first $19$ adaptive iterations.
In order to show the efficiency of {\bf Adaptive Algorithm $C$} more clearly, we compare
the results with those obtained with direct AFEM.
\begin{figure}[ht]
\centering
\includegraphics[width=7cm,height=7cm]{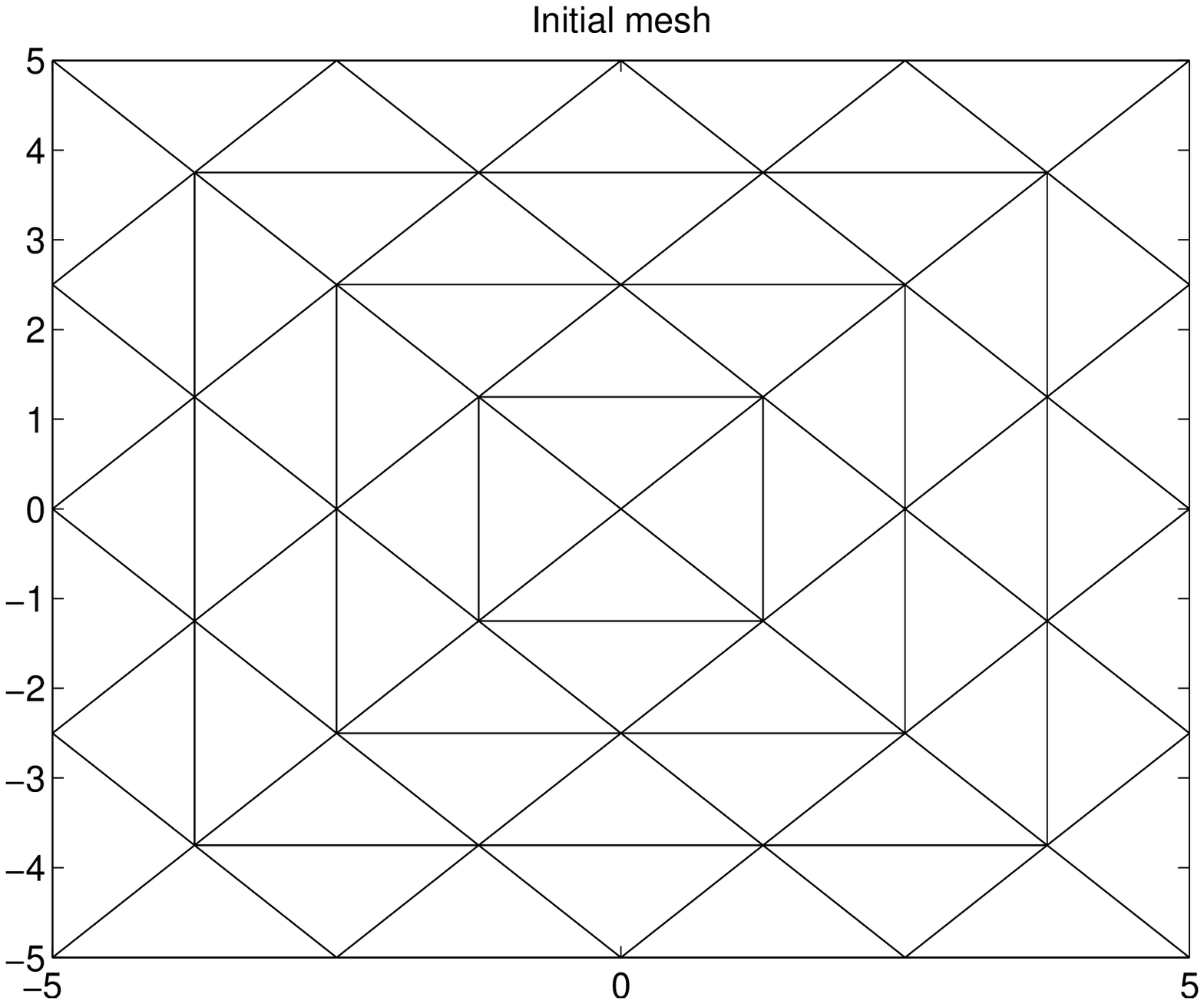}
\includegraphics[width=7cm,height=7cm]{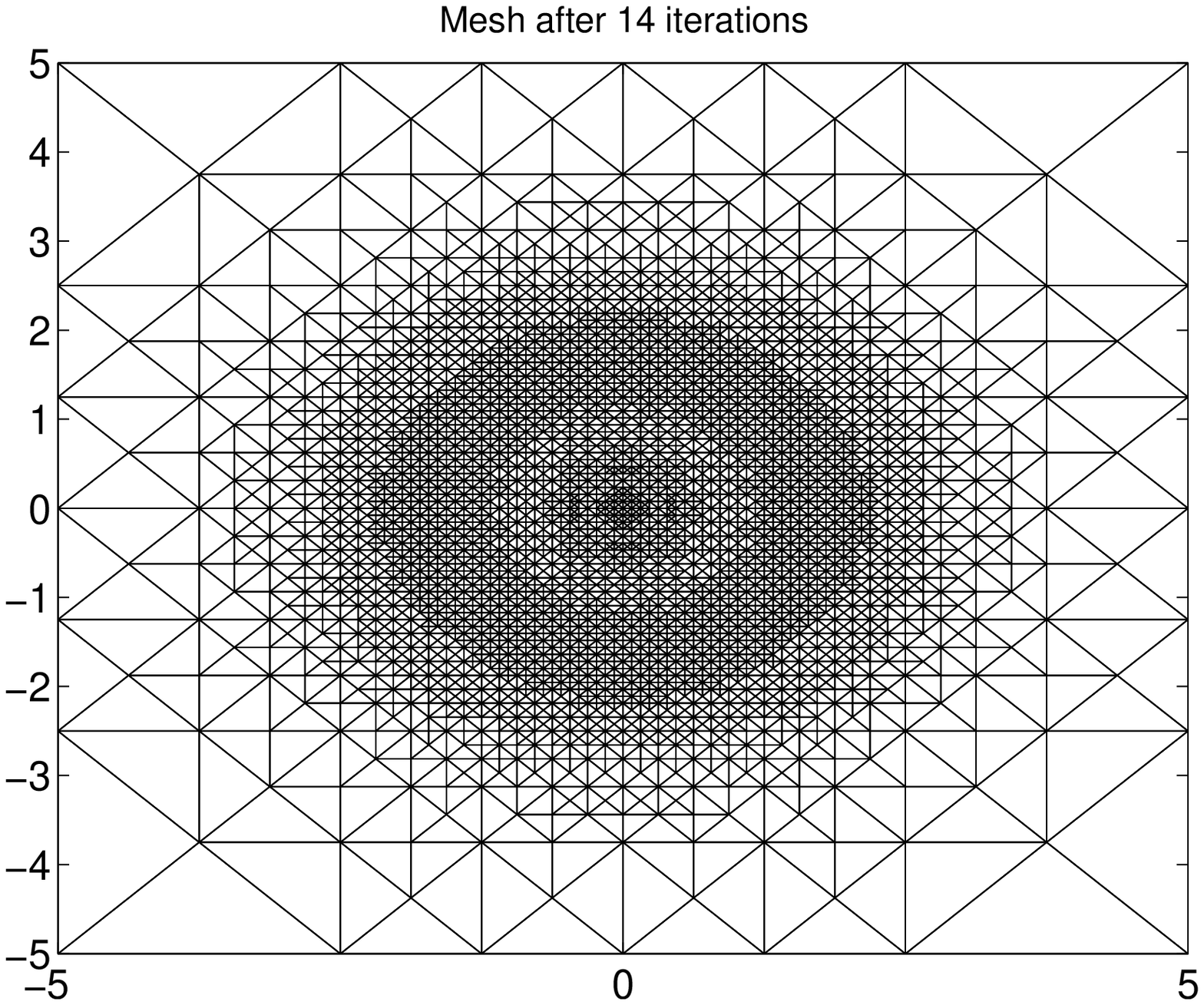}
\caption{The initial triangulation and the one after 14 adaptive iterations for
Example 1} \label{Mesh_AFEM_Exam_1}
\end{figure}
\begin{figure}[ht]
\centering
\includegraphics[width=7cm,height=7cm]{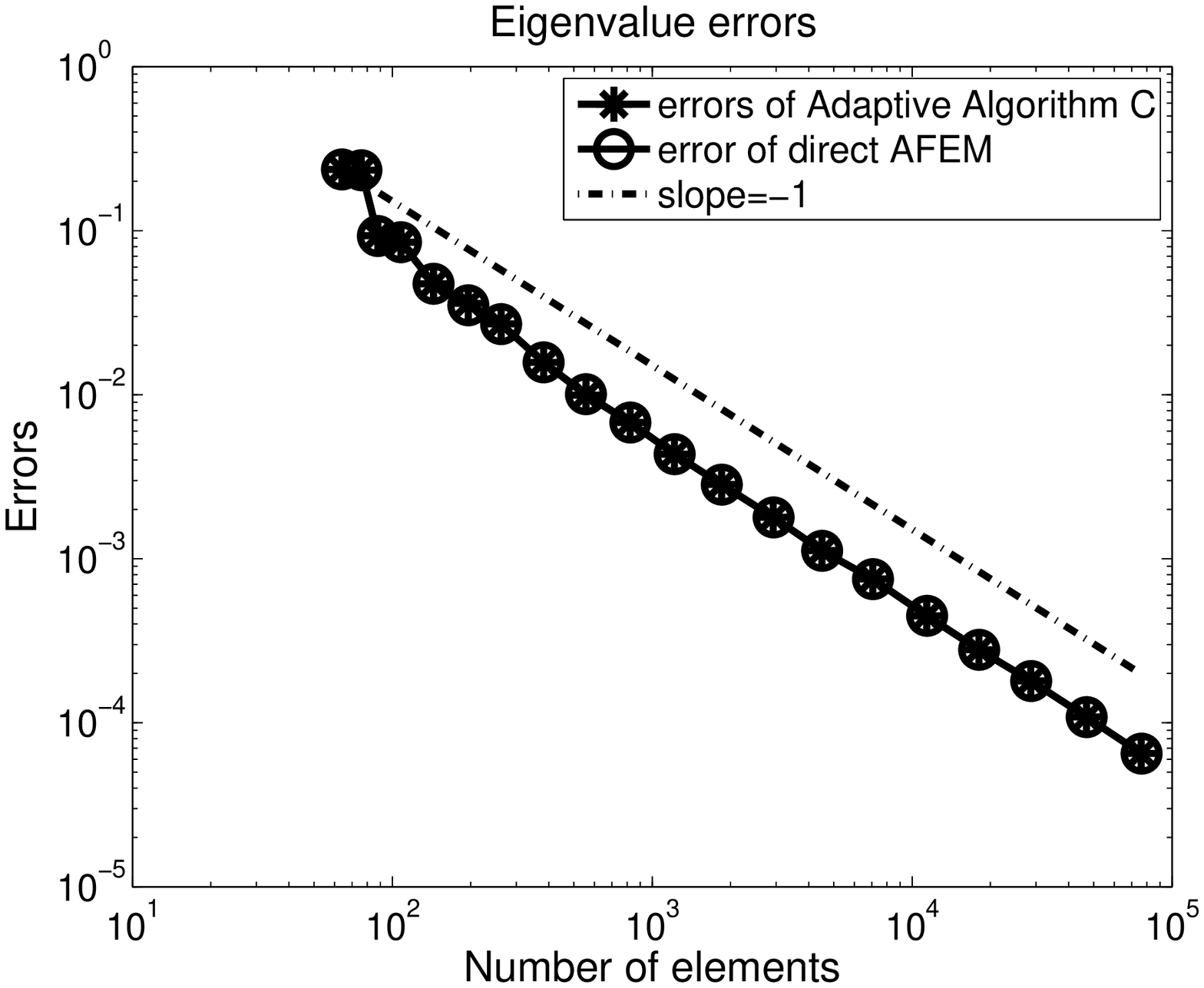}
\includegraphics[width=7cm,height=7cm]{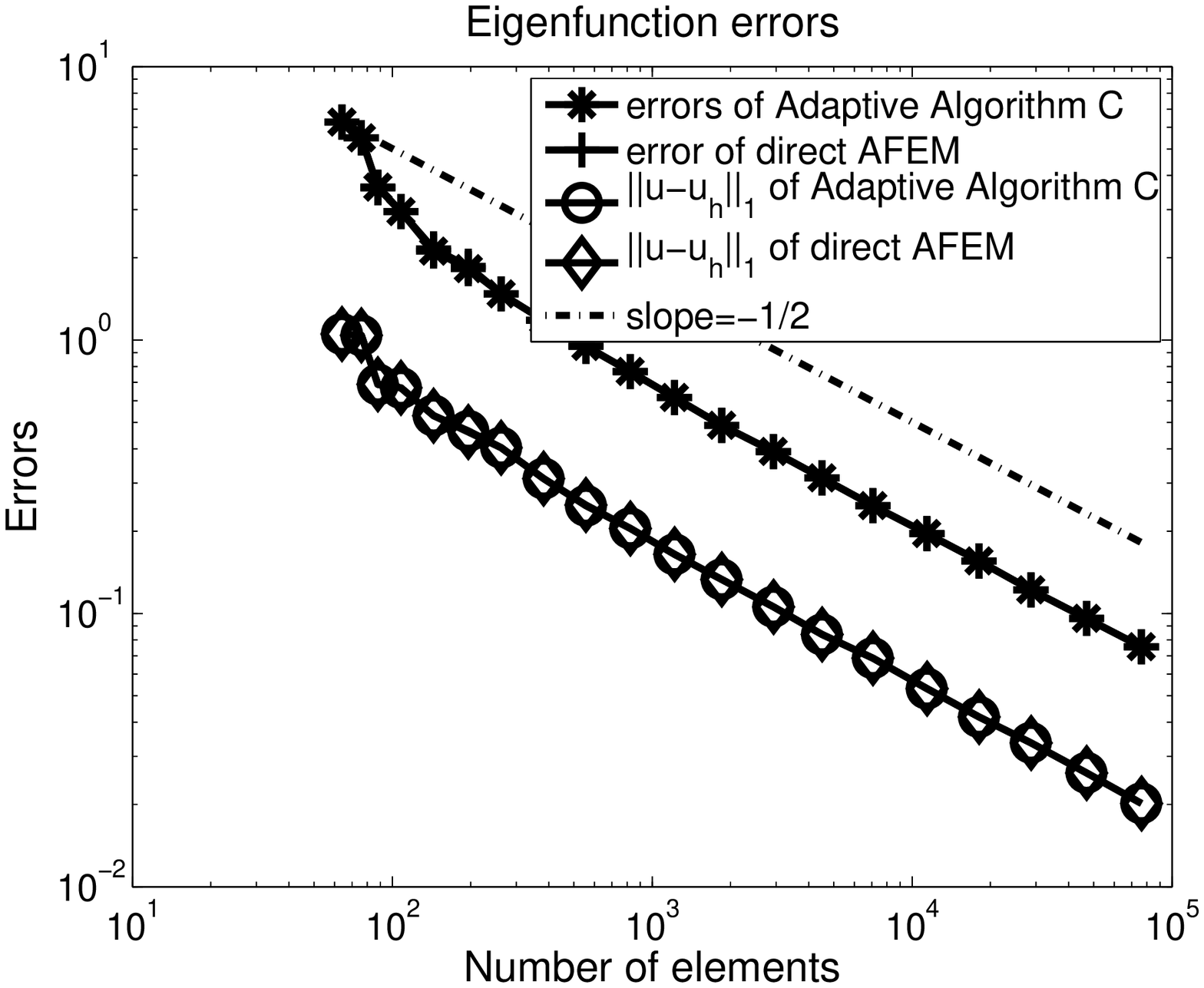}
\caption{The errors of the eigenvalue and the associated eigenfunction approximations by {\bf Adaptive Algorithm $C$}
and direct AFEM for Example 1} \label{Convergence_AFEM_Exam_1_First}
\end{figure}
It is observed from Figures \ref{Convergence_AFEM_Exam_1_First},
 the approximations of eigenvalue as well as eigenfunction approximations have the
 optimal convergence rate which coincides with our theory.

{\bf Example 2.}
In the second example, we consider the Laplace eigenvalue problem on the $L$-shape domain
\begin{equation}\label{eigenproblem_Exam_2}
\left\{
\begin{array}{rcl}
-\Delta u &=&\lambda u\ \ \ \ {\rm in}\ \Omega,\\
u&=&0\ \ \ \ \ \ {\rm on}\ \partial\Omega,\\
\|u\|_{0,\Omega}&=&1,
\end{array}
\right.
\end{equation}
where $\Omega=(-1,1)\times(-1,1)\backslash[0, 1)\times (-1, 0]$.
Since $\Omega$ has a reentrant corner, eigenfunctions with singularities are expected. The
convergence order for eigenvalue approximations is less than $2$ by the linear finite element method
 which is the order predicted by the theory for regular eigenfunctions.

First, we investigate the numerical results for the first eigenvalue approximations.
Since the exact eigenvalue is not known, we choose an adequately accurate approximation
$\lambda = 9.6397238440219$ as the exact first eigenvalue for our numerical tests.
We give the numerical results for the first eigenpair approximation of {\bf Adaptive Algorithm $C$}
with the parameter $\theta = 0.4$.
Figure \ref{Mesh_AFEM_Exam_2} shows the initial triangulation and the triangulation after $12$
adaptive iterations. Figure \ref{Convergence_AFEM_Exam_2_First}
 gives the corresponding numerical results for the first $20$ adaptive iterations.
In order to show the efficiency of {\bf Adaptive Algorithm $C$} more clearly, we compare
the results with those obtained by direct AFEM.
\begin{figure}[ht]
\centering
\includegraphics[width=7cm,height=6.8cm]{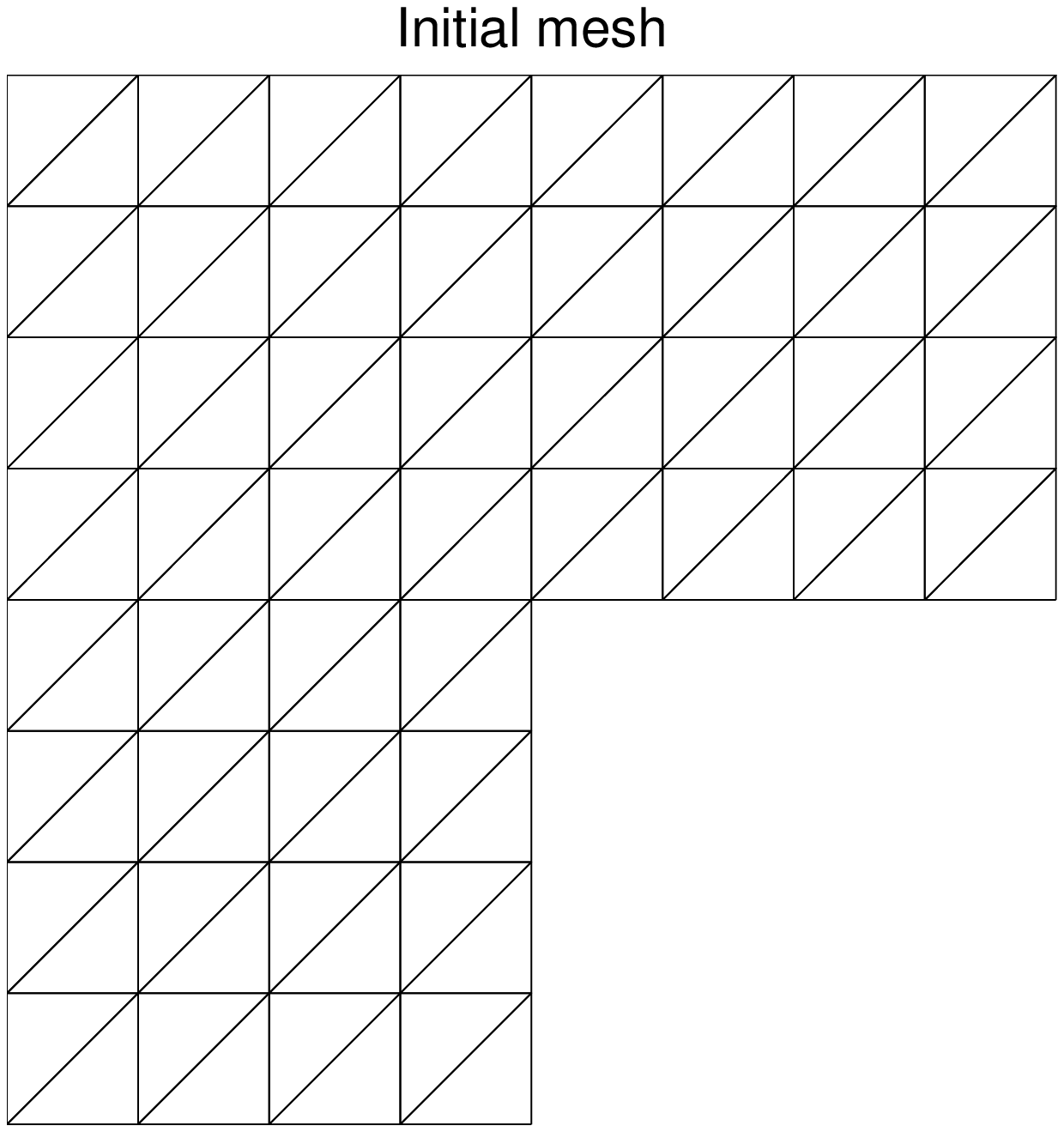}
\includegraphics[width=7cm,height=7cm]{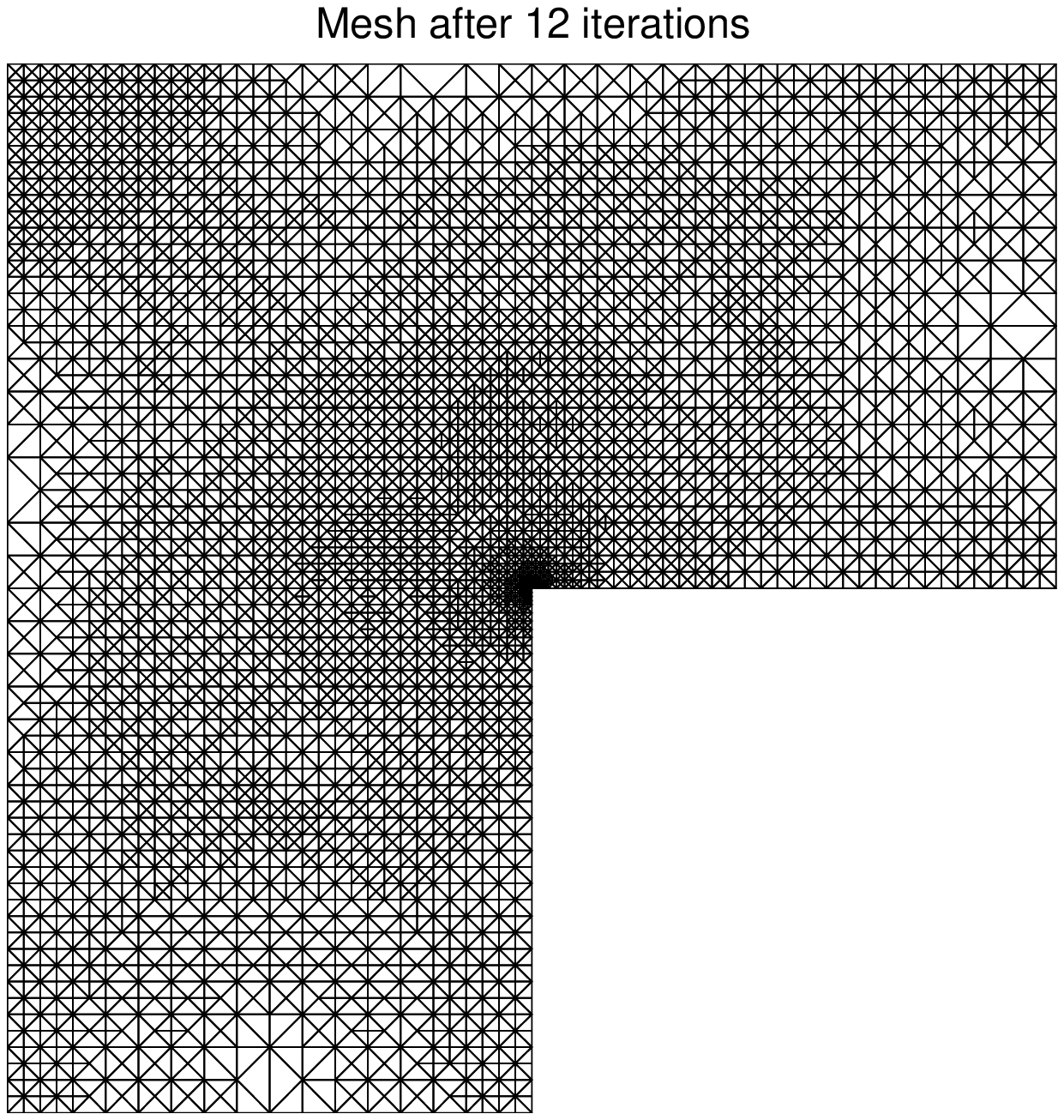}
\caption{The initial triangulation and the one after 12 adaptive iterations for
Example 2} \label{Mesh_AFEM_Exam_2}
\end{figure}
\begin{figure}[ht]
\centering
\includegraphics[width=7cm,height=7cm]{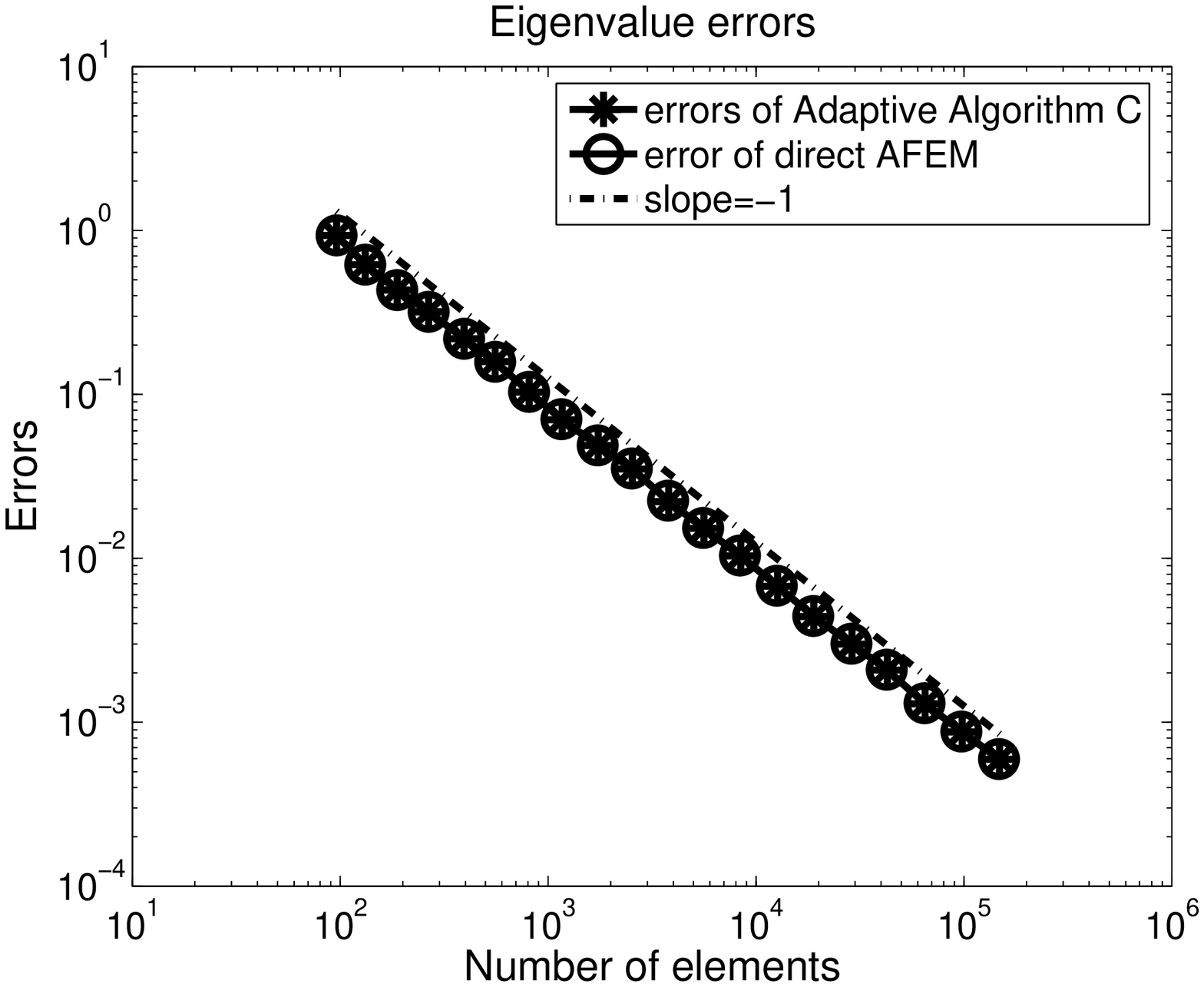}
\includegraphics[width=7cm,height=7cm]{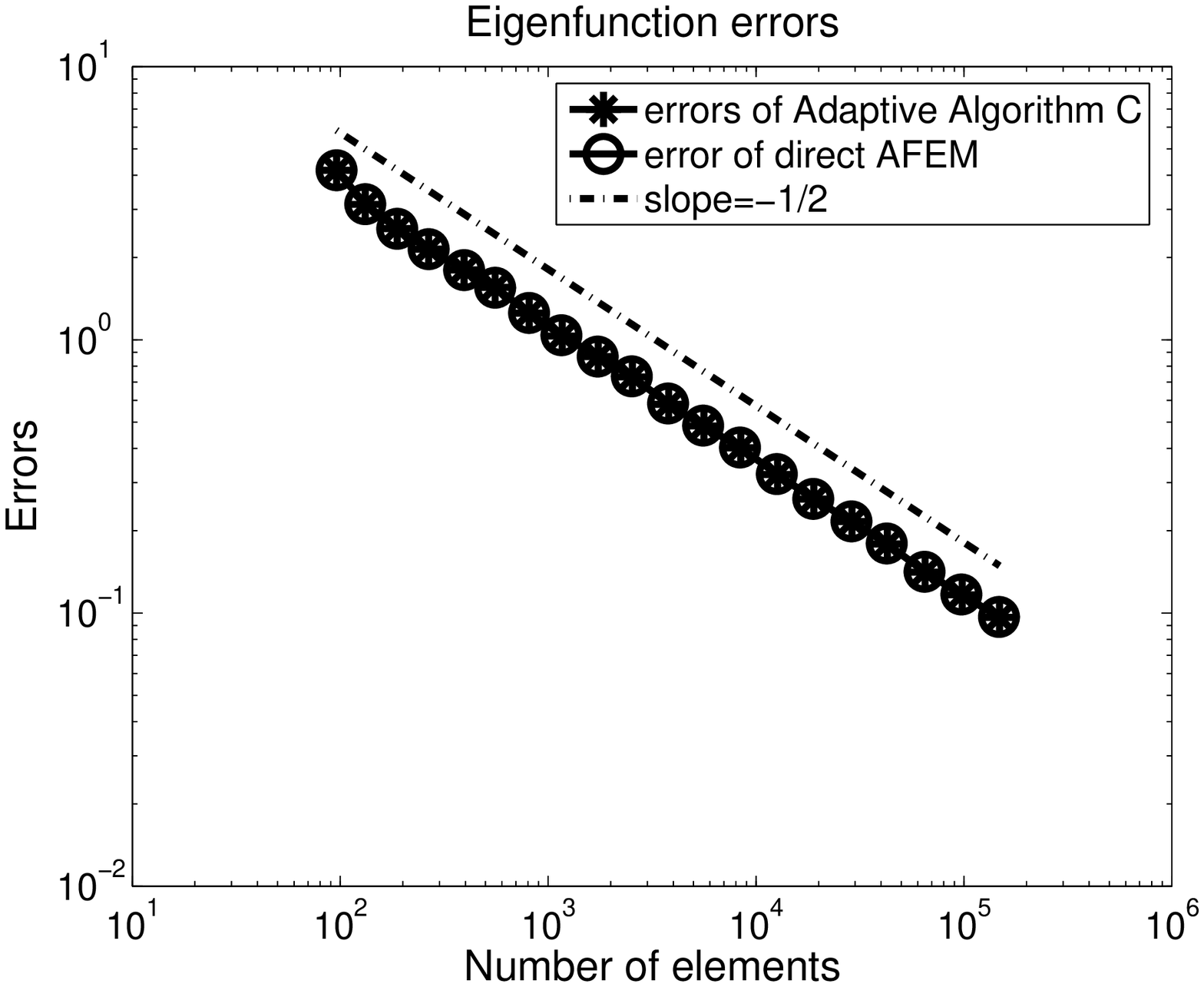}
\caption{The errors of the smallest eigenvalue approximations and the a
posteriori errors of the associated eigenfunction approximations by
{\bf Adaptive Algorithm $C$}
and direct AFEM for Example 2} \label{Convergence_AFEM_Exam_2_First}
\end{figure}
We also test {\bf Adaptive Algorithm $C$} for $5$ smallest eigenvalue approximations and their associated
eigenfunction approximations.  Figure \ref{Convergence_AFEM_Exam_2_5_Small} shows the corresponding a posteriori
error estimator $\eta_h(u_h,\Omega)$ produce by {\bf Adaptive Algorithm $C$} and direct AFEM.
\begin{figure}[ht]
\centering
\includegraphics[width=7cm,height=7cm]{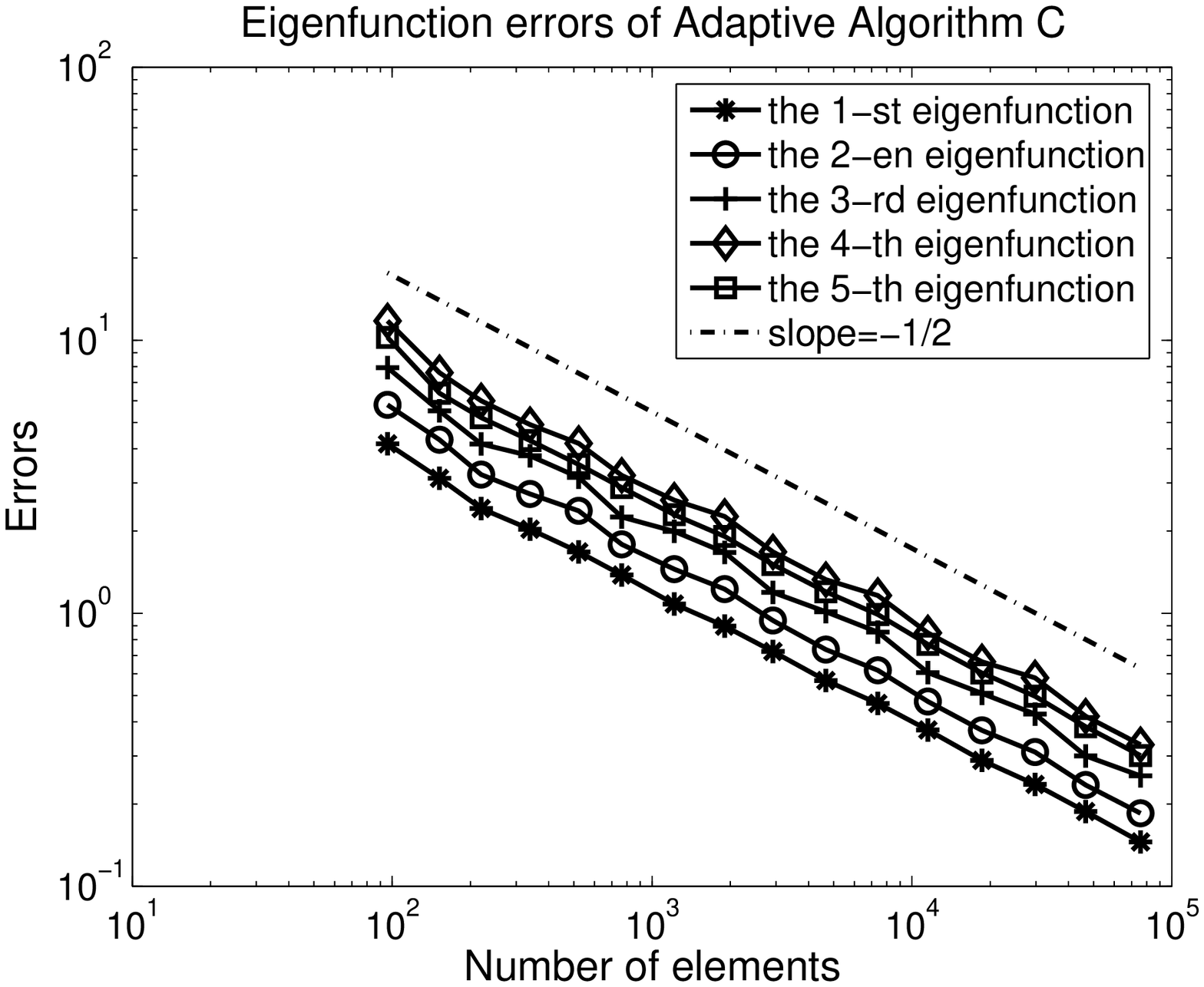}
\includegraphics[width=7cm,height=7cm]{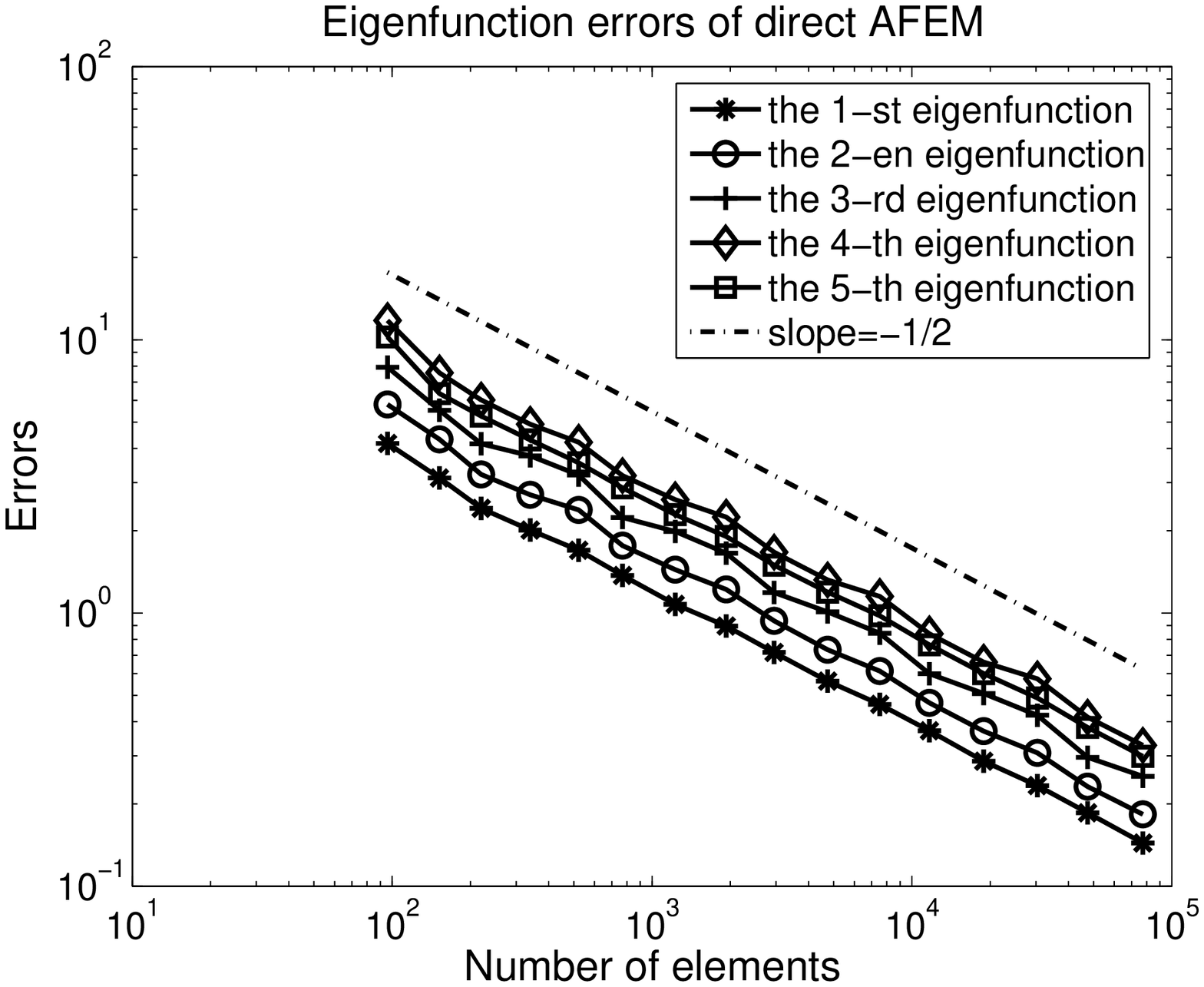}
\caption{The a posteriori error estimates of the eigenfunction approximations by {\bf Adaptive Algorithm $C$}
and direct AFEM for Example 2} \label{Convergence_AFEM_Exam_2_5_Small}
\end{figure}

From Figures \ref{Convergence_AFEM_Exam_2_First} and \ref{Convergence_AFEM_Exam_2_5_Small},
we can find the approximations of eigenvalues as well as eigenfunctions have optimal convergence
rate which coincides with our theory.

{\bf Example 3.} In this example, we consider the following second order elliptic
eigenvalue problem
\begin{equation}
\left\{
\begin{array}{rcl}
-\nabla\cdot(\mathcal{A}\nabla u) +\varphi u&=&\lambda u\ \ \ {\rm in}\ \Omega,\\
u&=&0\ \ \ \ \ {\rm on}\ \partial\Omega,\\
\|u\|_{0,\Omega}&=&1,
\end{array}
\right.
\end{equation}
with
\begin{equation*}
\mathcal{A} =
\left(
\begin{array}{cc}
1+(x_1-\frac{1}{2})^2 & (x_1-\frac{1}{2})(x_2-\frac{1}{2})\\
(x_1-\frac{1}{2})(x_2-\frac{1}{2}) & 1+(x_2-\frac{1}{2})^2
\end{array}
\right),
\end{equation*}
$\varphi=e^{(x_1-\frac{1}{2})(x_2-\frac{1}{2})}$ and
$\Omega=(-1,1)\times(-1,1)\backslash[0, 1)\times (-1, 0]$.

We first investigate the numerical results for the first eigenvalue approximations.
Since the exact eigenvalue is not known neither, we choose an adequately accurate
approximation $\lambda= 13.58258211870407$ as the exact eigenvalue for our numerical
 tests. We give the numerical results for the first eigenpair approximation by
  {\bf Adaptive Algorithm $C$} with the parameter $\theta = 0.4$. Figure \ref{Mesh_AFEM_Exam_3}
  shows the initial triangulation and the triangulation after $12$ adaptive iterations.
  Figure \ref{Convergence_AFEM_Exam_3_First} gives the corresponding numerical results for
  the first $18$ adaptive iterations. In order to show the efficiency of
  {\bf Adaptive Algorithm $C$} more clearly, we compare
the results with those obtained with direct AFEM.
\begin{figure}[ht]
\centering
\includegraphics[width=7cm,height=7cm]{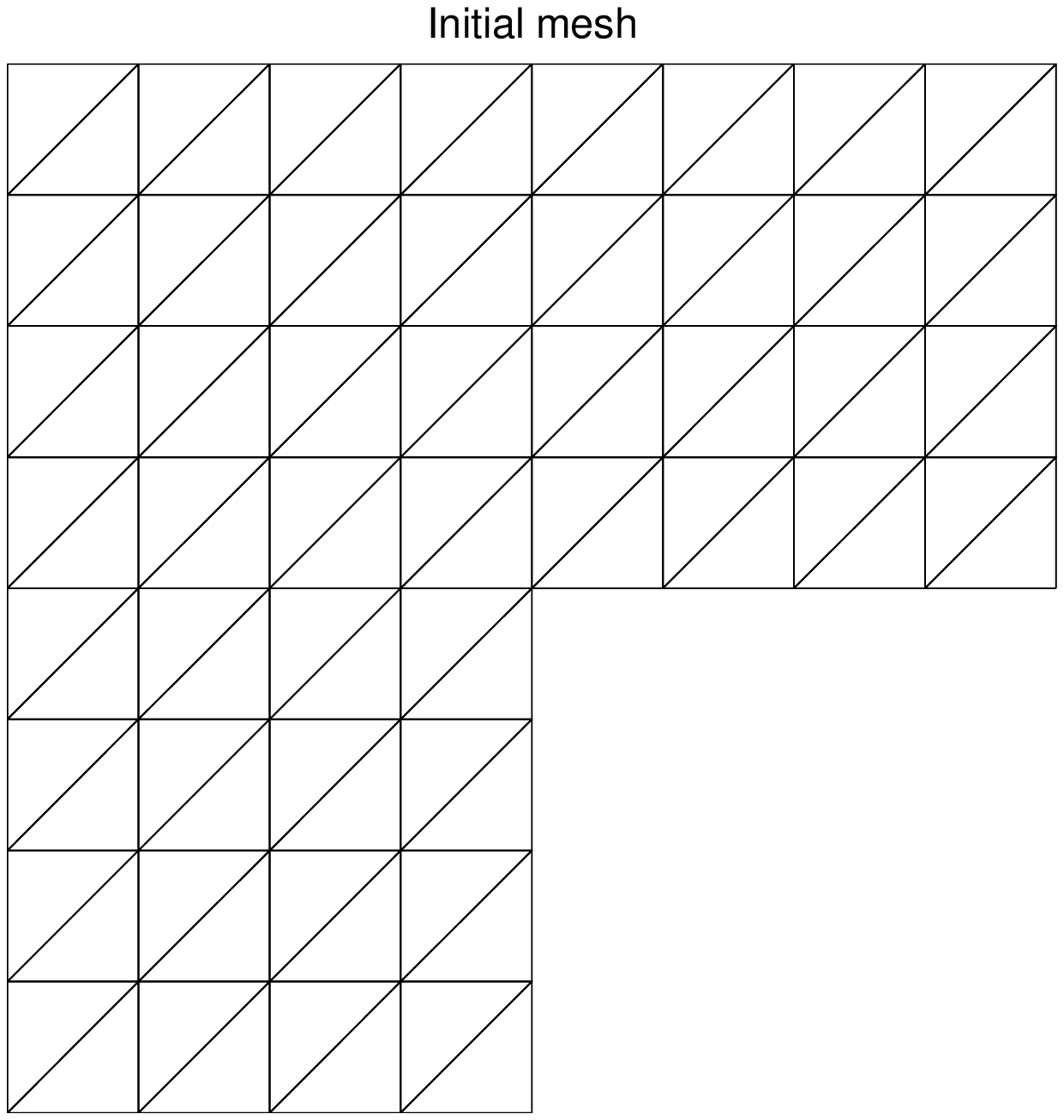}
\includegraphics[width=7cm,height=7cm]{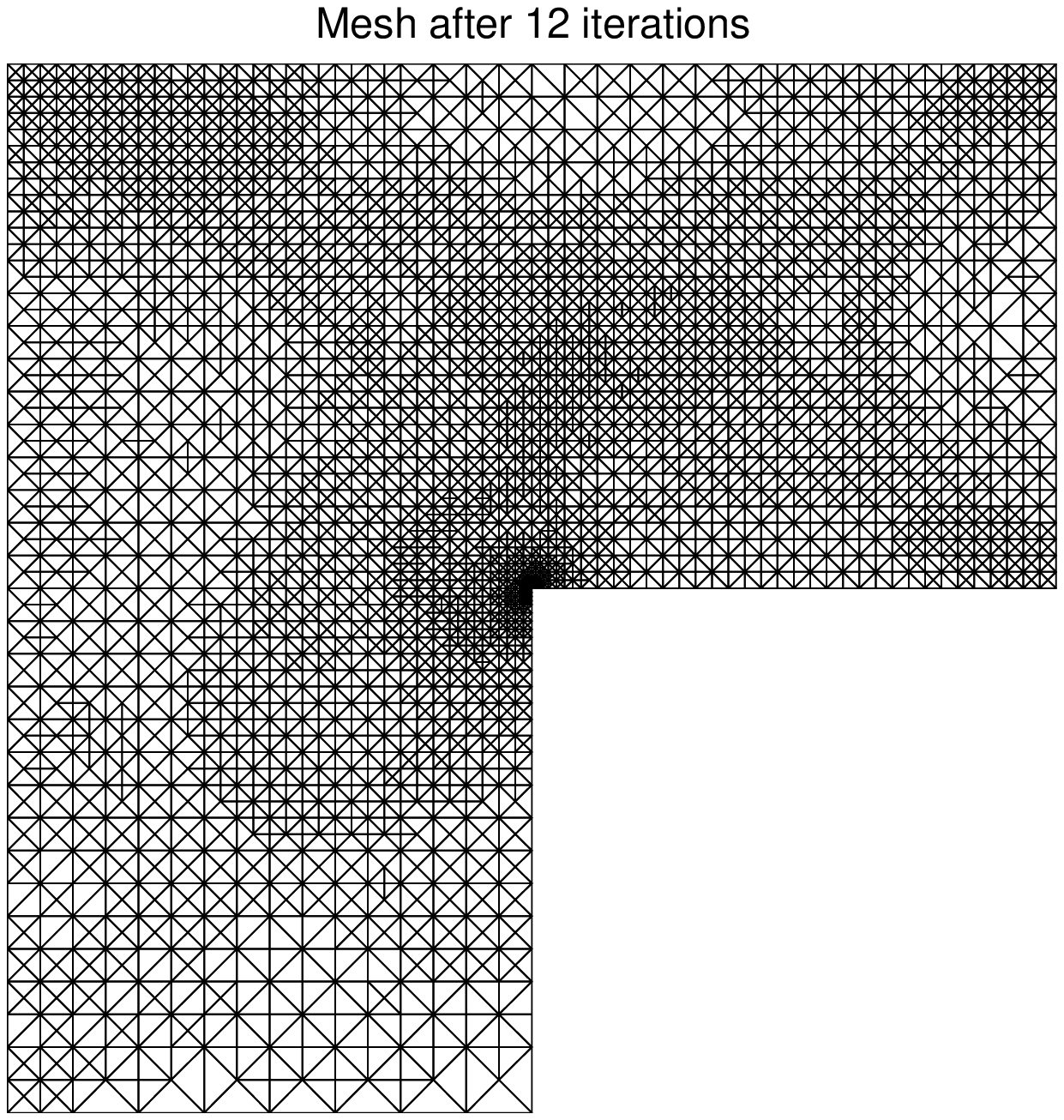}
\caption{The initial triangulation and the one after 12 adaptive iterations for
Example 3} \label{Mesh_AFEM_Exam_3}
\end{figure}
\begin{figure}[ht]
\centering
\includegraphics[width=7cm,height=7cm]{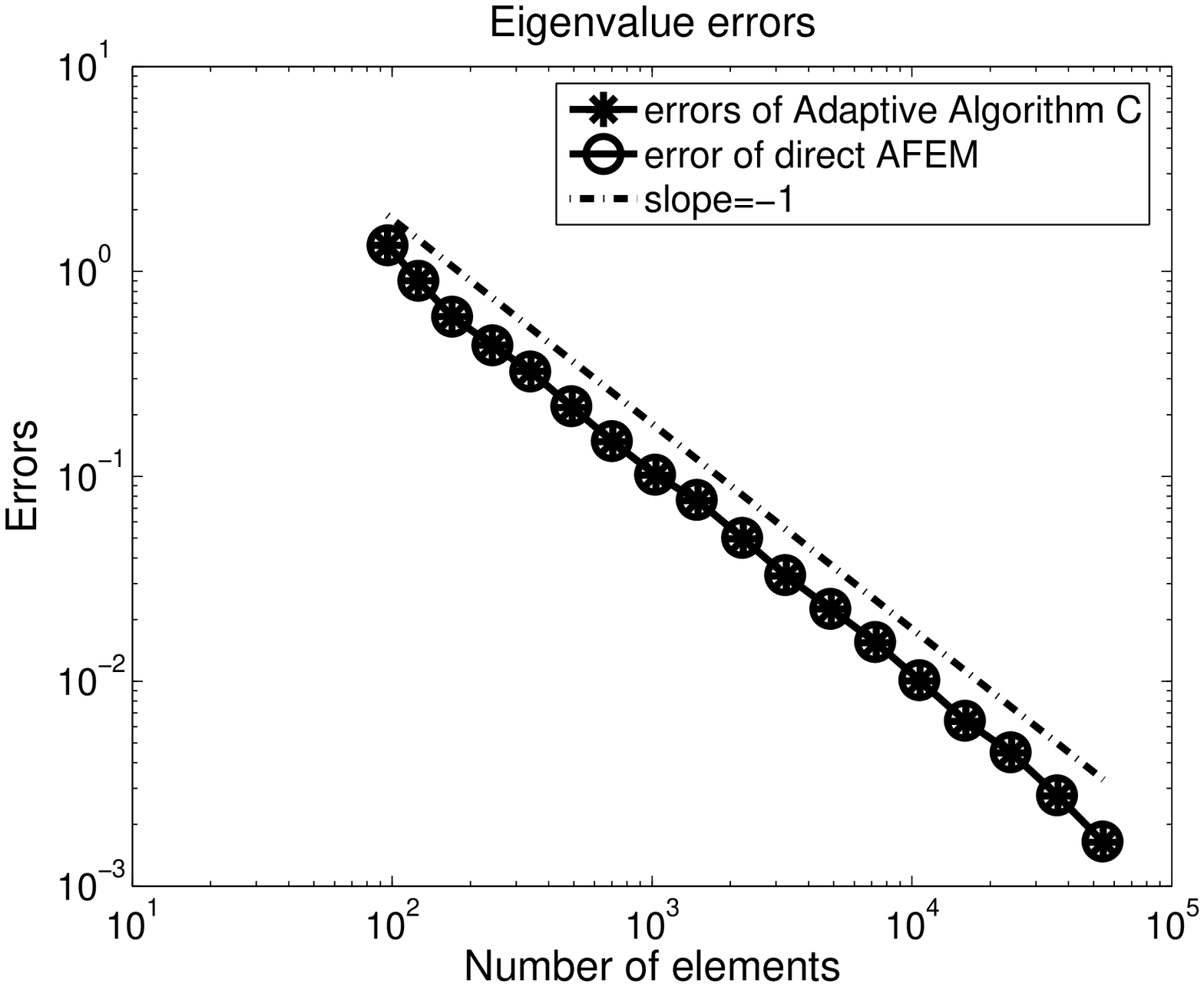}
\includegraphics[width=7cm,height=7cm]{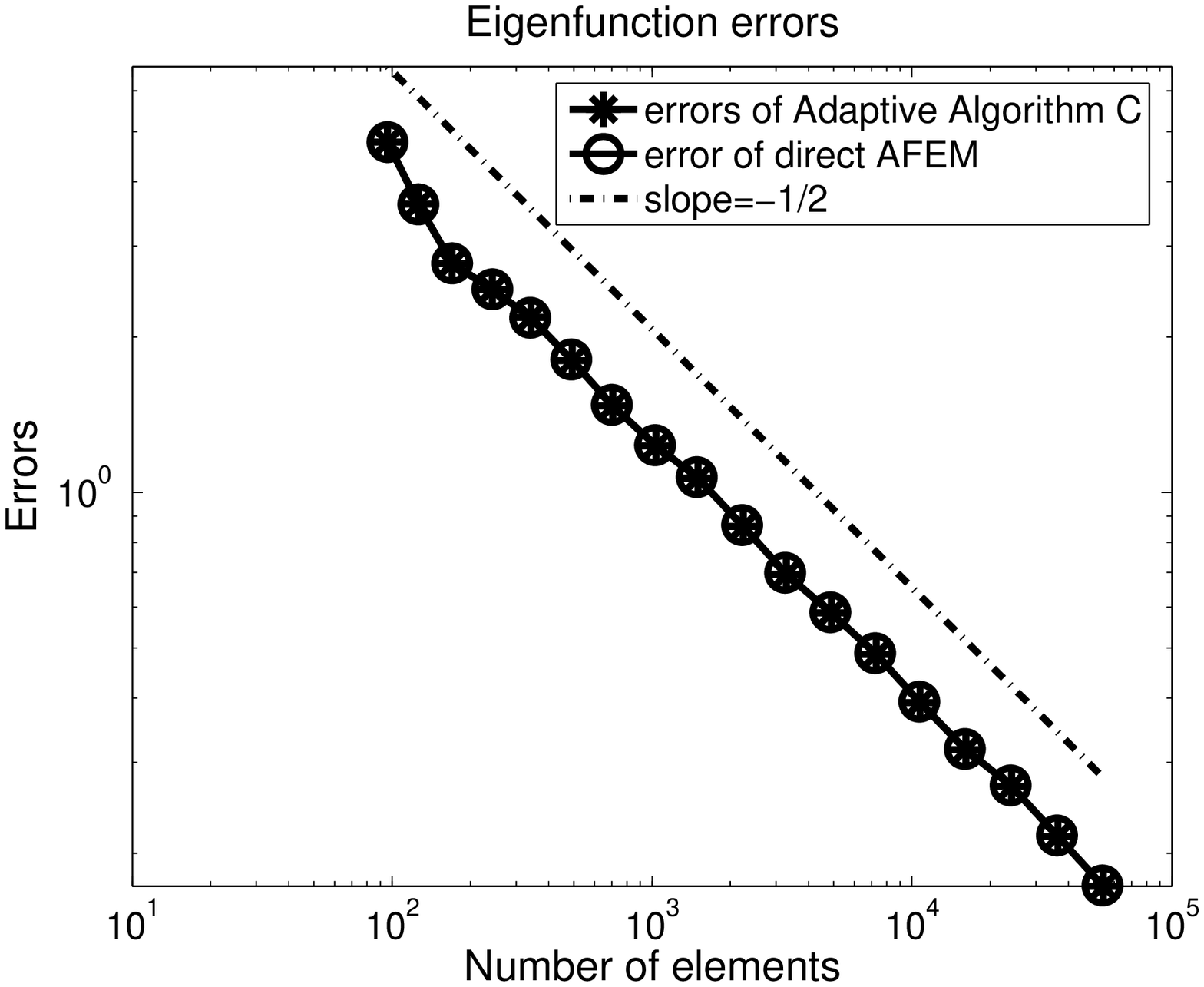}
\caption{The errors of the smallest eigenvalue and the associated eigenfunction approximations by {\bf Adaptive Algorithm $C$}
and direct AFEM for Example 3} \label{Convergence_AFEM_Exam_3_First}
\end{figure}

We also test {\bf Adaptive Algorithm $C$} for $5$ smallest eigenvalue approximations and their associated
eigenfunction approximations.  Figure \ref{Convergence_AFEM_Exam_3_5_Small} shows the a posteriori
error estimator $\eta_h(u_h,\Omega)$ produced by {\bf Adaptive Algorithm $C$} and direct AFEM.
\begin{figure}[ht]
\centering
\includegraphics[width=7cm,height=7cm]{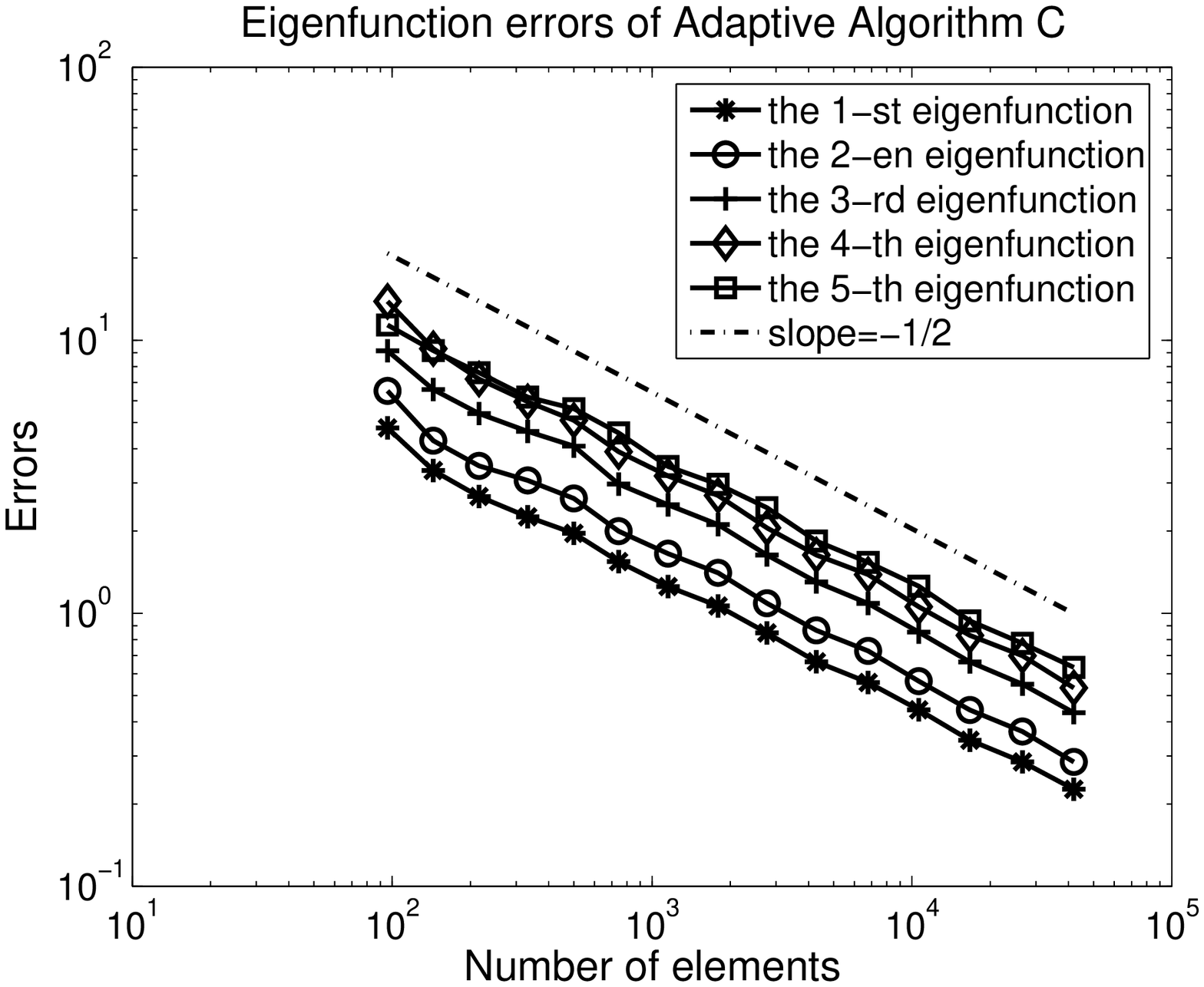}
\includegraphics[width=7cm,height=7cm]{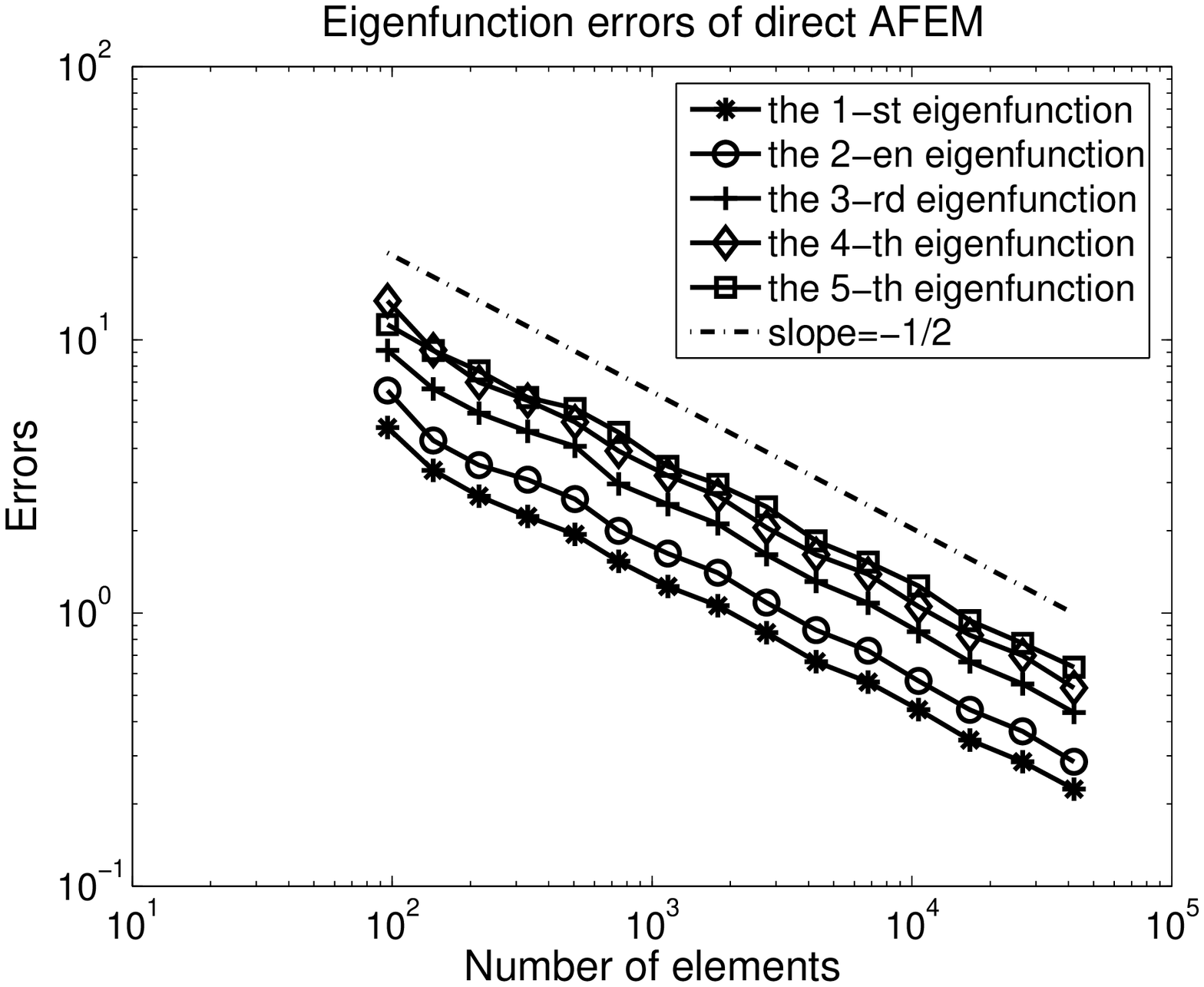}
\caption{The a posteriori error estimates of the eigenfunction approximations by {\bf Adaptive Algorithm $C$}
and direct AFEM for Example 3} \label{Convergence_AFEM_Exam_3_5_Small}
\end{figure}
From Figures \ref{Convergence_AFEM_Exam_3_First} and \ref{Convergence_AFEM_Exam_3_5_Small},
we can find the approximations of eigenvalues as well as eigenfunctions have optimal convergence
rate.

\section{Concluding remarks}
In this paper, we present a type of AFEM for eigenvalue problem based on multilevel
correction scheme. The convergence and optimal complexity have also been proved based on a relationship
between the eigenvalue problem and the associated boundary value problem (see Theorem \ref{trans}).
We also provide some numerical experiments to demonstrate the efficiency of the AFEM for eigenvalue problems.


\end{document}